\documentclass{amsart}
\usepackage{amssymb,euler}
\usepackage{amsmath}
\usepackage{amscd}
\newtheorem{prop}{Proposition}[section]
\newtheorem{defi}[prop]{Definition}
\newtheorem{conj}[prop]{Conjecture}

\newtheorem{lem}[prop]{Lemma}

\newtheorem{cor}[prop]{Corollary}

\newtheorem{remar}[prop]{Remark}

\newcommand{\Aut}{{\mathrm {Aut}}}
\def\id{\mathop{\mathrm{ id}}\nolimits}
\renewcommand{\Im}{{\mathrm {Im}}}
\newcommand{\ord}{{\mathrm {ord}}}

\newcommand{\Hom}{{\mathrm {Hom}}}

\newcommand{\Norm}{{\mathrm {Norm}}}

\newcommand{\tr}{{\mathrm {tr}}}
\newcommand{\Tr}{{\mathrm {Tr}}}

\newcommand{\Sym}{{\mathrm {Sym}}}

\newcommand{\crys}{{\mathrm {crys}}}

\newcommand{\spin}{{\mathrm {spin}}}
\newcommand{\dR}{{\mathrm {dR}}}

\newcommand{\Frob}{{\mathrm {Frob}}}

\newcommand{\Gal}{\mathrm {Gal}}

\newcommand{\res}{{\mathrm {res}}}

\newcommand{\AAA}{{\mathbb A}}
\newcommand{\CC}{{\mathbb C}}
\newcommand{\RR}{{\mathbb R}}

\newcommand{\QQ}{{\mathbb Q}}
\newcommand{\ZZ}{{\mathbb Z}}

\newcommand{\HH}{{\mathfrak H}}
\newcommand{\FFF}{{\mathcal F}}

\newcommand{\MMM}{{\mathfrak M}}

\newcommand{\FF}{{\mathbb F}}

\newcommand{\GL}{\mathrm {GL}}

\newcommand{\Sp}{\mathrm {Sp}}
\newcommand{\GSp}{\mathrm {GSp}}
\newcommand{\St}{\mathrm {St}}

\newcommand{\Qbar}{\overline{\mathbb Q}}
\newcommand{\rhobar}{\overline{\rho}}

\newcommand{\Fbar}{\overline{\mathbb F}}

\newcommand{\Wscr}{{\mathcal W}}
\newcommand{\disc}{{\mathrm disc}}
\newcommand{\bA}{{\mathbb A}}

\newcommand{\begeq}{\begin{equation}}

\newcommand{\zzendeq}{\end{equation}}
\newcommand{\R}{\mathbb{R}}
\newcommand{\Unn}{U_{\nu_1,\nu_2}}
\newcommand{\Q}{\mathbb{Q}}

\newcommand{\0}{\boldsymbol{0}}
\newcommand{\Z}{\mathbb{Z}}
\newcommand{\N}{\mathbb{N}}
\newcommand{\C}{\mathbb{C}}

\newcommand{\bQ}{\bold{Q}}

\newcommand{\cA}{{\mathcal A}}

\newcommand{\Ascr}{{\mathcal A}}
\newcommand{\by}{{\bf y}}
\newcommand{\bx}{{\bf x}}

\DeclareFontEncoding{OT2}{}{} 
  \newcommand{\textcyr}[1]{%
    {\fontencoding{OT2}\fontfamily{wncyr}\fontseries{m}\fontshape{n}%
     \selectfont #1}}
\newcommand{\Sha}{{\mbox{\textcyr{Sh}}}}

\newcommand{\ee}{{\mathbf e}}

\begin{document}
\title{Yoshida lifts and Selmer groups}
\author{Siegfried B\"ocherer}
\author{Neil Dummigan}
\author{Rainer Schulze-Pillot}
\date{September 28th, 2011.}
\subjclass{11F46, 11F67, 11F80, 11F33, 11G40}
\address{Kunzenhof 4B\\
79117 Freiburg, Germany}
\email{boecherer@math.uni-mannheim.de}
\address{University of Sheffield\\ School of Mathematics and Statistics\\
Hicks Building\\ Hounsfield Road\\ Sheffield, S3 7RH\\
U.K.}

\email{n.p.dummigan@shef.ac.uk}
 \address{FR 6.1 Mathematik\\Universit\"at des Saarlandes\\
 Postfach 151150\\D-66041 Saarbr\"ucken\\Germany}
 \email{schulzep@math.uni-sb.de}

\begin{abstract}
Let $f$ and $g$, of weights $k'>k\geq 2$, be normalised newforms for
$\Gamma_0(N)$, for square-free $N>1$, such that, for each
Atkin-Lehner involution, the eigenvalues of $f$ and $g$ are equal.
Let $\lambda\mid\ell$ be a large prime divisor of the algebraic part
of the near-central critical value $L(f\otimes g,\frac{k+k'-2}{2})$.
Under certain hypotheses, we prove that $\lambda$ is the modulus of
a congruence between the Hecke eigenvalues of a genus-two Yoshida
lift of (Jacquet-Langlands correspondents of) $f$ and $g$
(vector-valued in general), and a non-endoscopic genus-two cusp
form. In pursuit of this we also give a precise pullback formula for
a genus-four Eisenstein series, and a general formula for the
Petersson norm of a Yoshida lift.

Given such a congruence, using the $4$-dimensional $\lambda$-adic
Galois representation attached to a genus-two cusp form, we produce,
in an appropriate Selmer group, an element of order $\lambda$, as
required by the Bloch-Kato conjecture on values of $L$-functions.
\end{abstract}

\maketitle
\section{Introduction}
This paper is about congruences between modular forms, modulo large
prime divisors of normalised critical values of $L$-functions. The
first instance of this might be considered to be Ramanujan's
congruence modulo $691$ between the Hecke eigenvalues of the cusp
form $\Delta$ and an Eisenstein series of weight $12$ for
$SL_2(\ZZ)$, the prime $691$ occurring in the critical value
$\zeta(12)$. Congruences modulo $p$ between Eisenstein series and
cusp forms (now of weight $2$ and level $p$) were used by Ribet
\cite{R} to prove his converse to Herbrand's theorem. Interpreting
the congruence as a reducibility modulo $p$ of the $2$-dimensional
Galois representation attached to the cusp form, he used the
non-trivial extension of $1$-dimensional factors to construct
elements of order $p$ in the class group of $\QQ(\zeta_p)$. Mazur
and Wiles \cite{MW} developed this idea further in their proof of
Iwasawa's main conjecture. When Bloch and Kato \cite{BK} proved most
of their conjecture in the case of the Riemann zeta function, the
Mazur-Wiles theorem was the main ingredient.

Let $f$ and $g$, of weights $k'>k\geq 2$, be normalised newforms for
$\Gamma_0(N)$, for square-free $N>1$, such that, for each
Atkin-Lehner involution, the eigenvalues of $f$ and $g$ are equal.
Let $\lambda\mid\ell$ be a large prime divisor of the algebraic part
of the near-central critical value $L(f\otimes g,\frac{k+k'-2}{2})$
(or equivalently of its partner $L(f\otimes g,\frac{k+k'}{2})$). In
this paper, we seek a congruence modulo $\lambda$ between the Hecke
eigenvalues of a Yoshida lift $F=F_{f,g}$, and some other genus-$2$
Hecke eigenform $G$, of the same weight $\Sym^j\otimes
\det^{\kappa}$, where $j=k-2$ and $\kappa=2+\frac{k'-k}{2}$, and
level $\Gamma_0^{(2)}(N)$. (See \S 1.1 and later sections for
definitions and notation.) Proposition \ref{maincong} (and Corollary
\ref{Nprime}) is what we are able to prove. If $p$ is any prime
$p\nmid\ell N$ (where $\lambda\mid \ell$) and $\mu_G(p)$ is the
eigenvalue of the Hecke operator $T(p)$ acting on $G$, then the
congruence is
$$\mu_G(p)\equiv a_p(f)+p^{(k'-k)/2}a_p(g)\pmod{\lambda}.$$

Our proof is modelled on Katsurada's approach to proving congruences
between Saito-Kurokawa lifts and non-lifts \cite{Ka}, modulo
divisors of the near-central critical values of Hecke $L$-functions
of genus-$1$ cuspidal eigenforms of level $1$. Thus we consider a
``pullback formula'' for the restriction to $\HH_2\times\HH_2$ of a
genus-$4$, Eisenstein series (of weight $4$) to which a certain
differential operator has been applied. The coefficient of $F\otimes
F$ is some constant times a value of the standard $L$-function of
$F$, divided by the Petersson norm of $F$.

Section 6 contains a proof of the required pullback formula
(\ref{pullback4}) (derived, using also (15), from the more general
(9)), using differential operators from \cite{BoeFJII} and
\cite{BSY}, and taking care to determine the precise constants
occurring. Section 8 contains the proof of a formula for the
Petersson norm of the Yoshida lift $F$, generalising \cite{BS1},
which dealt with the analogous case where $k'=k=2$ and $F$ is
scalar-valued of weight $\kappa=2$. This proof uses another, more
subtle pullback formula
(\ref{ready_to_use_pullback_for_F_weight_2}), involving an
Eisenstein series of genus $4$ and weight $2$, also provided by
Section 6. The value $L(f\otimes g,\frac{k+k'}{2})$ thereby appears
as a factor in the formula for the Petersson norm of the Yoshida
lift, thus introducing $\lambda$ into a denominator in the pullback
formula referred to in the previous paragraph. The congruence is
then proved by some application of Hecke operators to both sides.

For this we need to know the integrality at $\lambda$ of the
left-hand-side (dealt with in Section 7), and, more problematically,
that some Fourier coefficient of a canonical scaling of the Yoshida
lift $F$ is not divisible by $\lambda$. (At this point Katsurada was
able to use an explicit formula for the Fourier coefficients of a
Saito-Kurokawa lift.) What we need on Fourier coefficients of
Yoshida lifts can be reduced to a weak condition on non-divisibility
by $\lambda$ of certain normalised $L$-values, in the case that $N$
is prime, Atkin-Lehner eigenvalue $\epsilon_N=-1$ and $k/2, k'/2$
are odd, using an averaging formula from \cite{BS5}. This condition
may be checked explicitly using a formula of Gross and Zagier. In
his thesis \cite{Ji}, Johnson Jia has worked out a different
approach to the problem of Fourier coefficients of Yoshida lifts mod
$\lambda$, in the scalar-valued case.

Brown \cite{Br} used the Galois interpretation of congruences (of
Hecke eigenvalues) between Saito-Kurokawa lifts and non-lifts, to
confirm a prediction of the Bloch-Kato conjecture. Likewise, in the
earlier sections of this paper we use congruences between Yoshida
lifts and non-lifts to produce non-zero elements of
$\lambda$-torsion in the appropriate Bloch-Kato Selmer group. (See
Proposition \ref{selmer}.) The required cohomology classes come from
non-trivial extensions inside the mod $\lambda$ reduction of
Weissauer's $4$-dimensional Galois representation attached to $G$.
This mod $\lambda$ representation is reducible thanks to the
congruence.

The work of Brown is easily extended to other (not necessarily
near-central) critical values of $L_f(s)$ if one assumes a
conjecture of Harder \cite{Ha, vdG} on the existence of congruences
involving vector-valued genus-$2$ cusp forms. It is not possible
likewise to extend the present work to other critical values of the
tensor-product $L$-function using genus-$2$ Siegel modular forms.
The problem is that we have two fixed parameters $k'$ and $k$, not
allowing any freedom to vary $j$ and $\kappa$. This is explained in
more detail at the end of \cite{Du2}.

M. Agarwal and K. Klosin, independently of us, at the suggestion of
C. Skinner, worked on using congruences between Yoshida lifts and
non-lifts to construct elements in Selmer groups, to support the
Bloch-Kato conjecture for tensor product $L$-functions at the near
central point \cite{AK}. Their approach to proving such congruences
is different, resulting in different conditions, and covers the
scalar-valued case ($k=2$). They use a Siegel-Eisenstein series with
a character, as in \cite{Br}, and take pains to avoid our assumption
(in Lemma \ref{endocap} and Proposition \ref{selmer}) that $\lambda$
is not a congruence prime for $f$ or $g$, at the cost of restricting
$k'$ to be $10$ or $14$.

{\em Acknowledgements. } We thank M. Agarwal, T. Berger, J.
Bergstr\"om, J. Jia, H. Katsurada, K. Klosin, C. Poor and D. Yuen
for helpful communications. We thank also M. Chida for pointing out
that \cite{We4} allows us to eliminate an unnecessary hypothesis.

\subsection{Definitions and notation} Let $\HH_n$ be the Siegel upper half plane of $n$ by $n$ complex
symmetric matrices with positive-definite imaginary part. Let
$\Gamma^{(n)}:=\Sp(n,\ZZ)=\Sp_{2n}(\ZZ)=\{M\in \GL_{2n}(\ZZ):
{}^tMJM=J\}$, where $J=\begin{pmatrix} 0_n& I_n\\-I_n &
0_n\end{pmatrix}$. For $M=\begin{bmatrix} A & B\\C &
D\end{bmatrix}\in \Gamma^{(n)}$ and $Z\in \HH_n$, let
$M(Z):=(AZ+B)(CZ+D)^{-1}$ and $J(M,Z):=CZ+D$. Let
$\Gamma_0^{(n)}(N)$ be the subgroup of $\Gamma^{(n)}$ defined by the
condition $N\mid C$. Let $V$ be the space of a finite-dimensional
representation $\rho$ of $\GL(n,\CC)$. A holomorphic function
$f:\HH_n\rightarrow V$ is said to belong to the space
$M_{\rho}(\Gamma_0^{(n)}(N))$ of Siegel modular forms of genus $n$
and weight $\rho$, for $\Gamma_0^{(n)}(N)$, if
$$f(M(Z))=\rho(J(M,Z))f(Z)\,\,\,\,\,\,\,\,\forall
M\in\Gamma_0^{(n)}(N),\, Z\in\HH_n.$$ In other words, $f|M=f$ for
all $M\in\Gamma_0^{(n)}(N)$, where
$(f|M)(Z):=\rho(J(M,Z))^{-1}f(M(Z))$ for $M\in \Sp_{2n}(\ZZ)$. Such
an $f$ has a Fourier expansion
$$f(Z)=\sum_{S\geq 0}a(S)\ee(\Tr(SZ))=\sum_{S\geq 0}a(S,f)\ee(\Tr(SZ)),$$
where the sum is over all positive semi-definite half-integral
matrices, and $\ee(z):=e^{2\pi i z}$.

Denote by $S_{\rho}(\Gamma_0^{(n)}(N))$, the subspace of cusp forms,
those that vanish at the boundary. They are also characterised by
the condition that, for all $M\in\Sp_{2n}(\ZZ)$, $a(S,f|M)=0$ unless
$S$ is positive-definite. When $\rho$ is of the special form
$\det^k\otimes\Sym^j(\CC^n)$ (where $\CC^n$ is the standard
representation of $\GL_n(\CC)$), the Petersson inner product will be
as in \S 2 of \cite{Koz}, and when also $n=2$, the Hecke operators
$T(m)$, for $(m,N)=1$, will be defined as in \S 2 of \cite{Ar},
replacing $\Sp_4(\ZZ)$ by $\Gamma_0^{(2)}(N)$. When $j=0$, we are
dealing with the usual scalar-valued Siegel cusp forms of weight
$k$. For a Hecke eigenform $F$, the incomplete spinor and standard
$L$-functions $L^{(N)}(F,s,\spin)$ and $L^{(N)}(F,s,\St)$ may be
defined in terms of Satake parameters as in \cite{An}, see also \S
20 of \cite{vdG}.
\section{Critical values of the tensor product $L$-function}
Let $f\in S_{k'}(\Gamma_0(N)), g\in S_k(\Gamma_0(N))$ be normalised
newforms (with $k'>k\geq 2$), $K$ some number field containing all
the Hecke eigenvalues of $f$ and $g$. Attached to $f$ is a
``premotivic structure'' $M_f$ over $\QQ$ with coefficients in $K$.
Thus there are $2$-dimensional $K$-vector spaces $M_{f,B}$ and
$M_{f,\dR}$ (the Betti and de Rham realisations) and, for each
finite prime $\lambda$ of $O_K$, a $2$-dimensional
$K_{\lambda}$-vector space $M_{f,\lambda}$, the $\lambda$-adic
realisation. These come with various structures and comparison
isomorphisms, such as $M_{f,B}\otimes_K K_{\lambda}\simeq
M_{f,\lambda}$. See 1.1.1 of \cite{DFG} for the precise definition
of a premotivic structure, and 1.6.2 of \cite{DFG} for the
construction of $M_f$, which uses the cohomology, with, in general,
non-constant coefficients, of modular curves, and pieces cut out
using Hecke correspondences.

On $M_{f,B}$ there is an action of $\Gal(\CC/\RR)$, and the
eigenspaces $M_{f,B}^{\pm}$ are $1$-dimensional. On $M_{f,\dR}$
there is a decreasing filtration, with $F^j$ a $1$-dimensional space
precisely for $1\leq j\leq k'-1$. The de Rham isomorphism
$M_{f,B}\otimes_K\CC\simeq M_{f,\dR}\otimes_K\CC$ induces
isomorphisms between $M_{f,B}^{\pm}\otimes\CC$ and
$(M_{f,\dR}/F)\otimes\CC$, where $F:=F^1=\ldots=F^{k'-1}$. Define
$\omega^{\pm}$ to be the determinants of these isomorphisms. These
depend on the choice of $K$-bases for $M_{f,B}^{\pm}$ and
$M_{f,\dR}/F$, so should be viewed as elements of
$\CC^{\times}/K^{\times}$. In exactly the same way there is also a
premotivic structure $M_g$, but since $k'>k$, it turns out that it
is the periods of $f$ that will show up in the formula for the
periods of the rank-$4$ premotivic structure $M_{f\otimes
g}:=M_f\otimes M_g$.

From the above properties of $M_f$ and $M_g$, one easily obtains the
following properties of $M_{f\otimes g}$. The eigenspaces
$M^{\pm}_{f\otimes g,B}$ are $2$-dimensional. On $M_{f\otimes
g,\dR}$ there is a decreasing filtration, with $F^t$ a
$2$-dimensional space precisely for $k\leq t\leq k'-1$. The de Rham
isomorphism $M_{f\otimes g,B}\otimes_K\CC\simeq M_{f\otimes
g,\dR}\otimes_K\CC$ induces an isomorphism between
$M^{\pm}_{f\otimes g,B}\otimes\CC$ and $(M_{f\otimes
g,\dR}/F')\otimes\CC$, where $F':=F^{k}=\ldots=F^{k'-1}$. Define
$\Omega^{\pm}\in\CC^{\times}/K^{\times}$ to be the determinants of
these isomorphisms.

For use in the next section, we shall choose an $O_K$-submodule
$\MMM_{f,B}$, generating $M_{f,B}$ over $K$, but not necessarily
free, and likewise an $O_K[1/S]$-submodule $\MMM_{f,\dR}$,
generating $M_{f,\dR}$ over $K$, where $S$ is the set of primes
dividing $N(k'!)$. We take these as in 1.6.2 of \cite{DFG}. They are
part of the ``$S$-integral premotivic structure'' associated to $f$,
and are defined using integral models and integral coefficients.
Actually, it will be convenient to enlarge $S$ so that $O_K[1/S]$ is
a principal ideal domain, then replace $\MMM_{f,B}$ and
$\MMM_{f,\dR}$ by their tensor products with the new $O_K[1/S]$.
These will now be free, as will be any submodules, and the quotients
we consider. Choosing bases, and using these to calculate the above
determinants, we pin down the values of $\omega^{\pm}$ (up to
$S$-units). Setting $\MMM_{f\otimes
g,B}:=\MMM_{f,B}\otimes\MMM_{g,B}$ and $\MMM_{f\otimes
g,\dR}:=\MMM_{f,\dR}\otimes\MMM_{g,\dR}$, similarly we pin down
$\Omega^{\pm}$ (up to $S$-units). We just have to imagine not
including in $S$ any prime we care about.

For each prime $\lambda$ of $O_K$ (say $\lambda\mid\ell$), the
$\lambda$-adic realisation $M_{f,\lambda}$ comes with a continuous
linear action of $\Gal(\Qbar/\QQ)$. For each prime number $p\neq
\ell$, the restriction to $\Gal(\Qbar_p/\QQ_p)$ may be used to
define a local $L$-factor
$[\det(I-\Frob_p^{-1}p^{-s}|M_{f,\lambda}^{I_p})]^{-1}$ (which turns
out to be independent of $\lambda$), and the Euler product is
precisely $L_f(s)$. (Here $I_p$ is an inertia subgroup at $p$, and
$\Frob_p$ is a Frobenius element reducing to the generating
$p^{\mathrm{th}}$-power automorphism in $\Gal(\Fbar_p/\FF_p)$.) In
exactly the same way we may use the Galois representation
$M_{f\otimes g,\lambda}=M_{f,\lambda}\otimes M_{g,\lambda}$ to
define the tensor product $L$-function $L_{f\otimes g}(s)$.
According to Deligne's conjecture \cite{De}, for each integer $t$ in
the critical range $k\leq t\leq k'-1$,
$$L_{f\otimes g}(t)/\Omega(t)\in K,$$ where $\Omega(t)=(2\pi
i)^{2t}\Omega^{(-1)^t}$ is the Deligne period for the Tate twist
$M_{f\otimes g}(t)$.

It is more convenient to use $\langle f,f\rangle$ than
$\Omega^{\pm}$, so we consider the relation between the two.
Calculating as in (5.18) of \cite{Hi}, using Lemma 5.1.6 of
\cite{De} and the latter part of 1.5.1 of \cite{DFG}, one recovers
the well-known fact that, up to $S$-units,
\begin{equation}\label{omeg}\langle
f,f\rangle=i^{k'-1}\omega^+\omega^- c(f),\end{equation} where
$c(f)$, the ``cohomology congruence ideal'', is, as the cup-product
of basis elements for $\MMM_{f,B}$, an integral ideal. Moreover,
calculating as in Lemma 5.1 of \cite{Du}, we find that
$$\Omega^+=\Omega^-=2(2\pi i)^{1-k}\omega^+\omega^-.$$
Hence Deligne's conjecture is equivalent to
$$\frac{L_{f\otimes g}(t)}{\pi^{2t-(k-1)}\langle f,f\rangle}\,\in
K$$ (for each integer $k\leq t\leq k'-1$). This is known to be true,
using Shimura's Rankin-Selberg integral for $L_{f\otimes g}(s)$
\cite{Sh}. In the next section we consider the integral refinement
of Deligne's conjecture.

\section{The Bloch-Kato conjecture}
We shall need the elements $\MMM_{f,\lambda}$ of the $S$-integral
premotivic structure, for each prime $\lambda$ of $O_K$. These are
as in 1.6.2 of \cite{DFG}. For each $\lambda$, $\MMM_{f,\lambda}$ is
a $\Gal(\Qbar/\QQ)$-stable $O_{\lambda}$-lattice in $M_{f,\lambda}$.
Similarly we have $\MMM_{g,\lambda}$, and $\MMM_{f\otimes
g,\lambda}:=\MMM_{f,\lambda}\otimes\MMM_{g,\lambda}$.

Let $A_{\lambda}:=M_{f\otimes g,\lambda}/\MMM_{f\otimes g,\lambda}$,
and $A[\lambda]:=A_{\lambda}[\lambda]$ the $\lambda$-torsion
subgroup. Let $\check{A_{\lambda}}:=\check{M}_{f\otimes
g,\lambda}/\check{\MMM}_{f\otimes g,\lambda}$, where
$\check{M}_{f\otimes g,\lambda}$ and $\check{\MMM}_{f\otimes
g,\lambda}$ are the vector space and $O_{\lambda}$-lattice dual to
$M_{f\otimes g,\lambda}$ and $\MMM_{f\otimes g,\lambda}$
respectively, with the natural $\Gal(\Qbar/\QQ)$-action. Let
$A:=\oplus_{\lambda}A_{\lambda}$, etc.

Following \cite{BK} (Section 3), for $p\neq \ell$ (where
$\lambda\mid\ell$, including $p=\infty$) let
$$H^1_f(\QQ_p,M_{f\otimes g,\lambda}(t))=\ker\bigl(H^1(D_p,M_{f\otimes g,\lambda}(t))
\rightarrow H^1(I_p,M_{f\otimes g,\lambda}(t))\bigr).$$ Here $D_p$
is a decomposition subgroup at a prime above~$p$, $I_p$ is the
inertia subgroup, and $M_{f\otimes g,\lambda}(t)$ is a Tate twist of
$M_{f\otimes g,\lambda}$, etc. The cohomology is for continuous
cocycles and coboundaries. For $p=\ell$ let
$$H^1_f(\QQ_{\ell},M_{f\otimes g,\lambda}(t))=
\ker\bigl(H^1(D_{\ell},M_{f\otimes g,\lambda}(t))\rightarrow
H^1(D_{\ell},M_{f\otimes g,\lambda}(t)\otimes_{\QQ_{\ell}}
B_{\crys})\bigr).$$ (See Section 1 of \cite{BK} or \S 2 of
\cite{Fo1} for the definition of Fontaine's ring $B_{\crys}$.) Let
$H^1_f(\QQ,M_{f\otimes g,\lambda}(t))$ be the subspace of those
elements of $H^1(\QQ,M_{f\otimes g,\lambda}(t))$ that, for all
primes~$p$, have local restriction lying in $H^1_f(\QQ_p,M_{f\otimes
g,\lambda}(t))$. There is a natural exact sequence
$$\begin{CD}0@>>>\MMM_{f\otimes g,\lambda}(t)@>>>M_{f\otimes g,\lambda}(t)@>\pi>>A_{\lambda}(t)@>>>0
\end{CD}.$$
Let $H^1_f(\QQ_p,A_{\lambda}(t))=\pi_*H^1_f(\QQ_p,M_{f\otimes
g,\lambda}(t))$. Define the ${\lambda}$-Selmer group
$H^1_f(\QQ,A_{\lambda}(t))$ to be the subgroup of elements of
$H^1(\QQ,A_{\lambda}(t))$ whose local restrictions lie in
$H^1_f(\QQ_p,A_{\lambda}(t))$ for all primes $p$. Note that the
condition at $p=\infty$ is superfluous unless $\ell=2$. Define the
Shafarevich-Tate group
$$\Sha(t)=\bigoplus_{\lambda}{H^1_f(\QQ,A_{\lambda}(t))
\over\pi_*H^1_f(\QQ,M_{f\otimes g,\lambda}(t))}.$$

Tamagawa factors $c_p(t)$ may be defined as in 11.3 of \cite{Fo}
(where the notation is $\mathrm{Tam}^0\ldots$). The $\lambda$ part
(for $\ell\neq p$) is trivial if $A_{\lambda}^{I_p}$ is divisible
(for example if $p\nmid N$). The following is equivalent to the
relevant cases of the Fontaine-Perrin-Riou extension of the
Bloch-Kato conjecture to arbitrary weights (i.e. not just points
right of the centre) and not-necessarily-rational coefficients.
(This follows from 11.4 of \cite{Fo}.) Note that by ``$\#$'' we
really mean the Fitting ideal.

\begin{conj}\label{bk}
Suppose that $k\leq t\leq k'-1$. Then we have the following equality
of fractional ideals of $O_K[1/S]$:
\begin{equation}\label{blka}\frac{L_{f\otimes g}(t)}{\Omega(t)}=\frac{\prod_{p\leq
\infty}c_p(t)~\#\Sha(t)}{\#H^0(\QQ,A(t))\#H^0(\QQ,\check{A}(1-t))}.\end{equation}
\end{conj}
In other words,
\begin{equation}\label{blka2}\frac{L_{f\otimes g}(t)}{\pi^{2t-(k-1)}\langle f,f\rangle}=\frac{\prod_{p\leq
\infty}c_p(t)~\#\Sha(t)}{\#H^0(\QQ,A(t))\#H^0(\QQ,\check{A}(1-t))c(f)}.
\end{equation}

Let $f=\sum a_n(f)q^n$. Let
$\rho_f:\Gal(\Qbar/\QQ)\rightarrow\Aut(M_{f,\lambda})$ be the
$2$-dimensional $\lambda$-adic Galois representation attached to
$f$. Let $\rhobar_f$ be its reduction $\pmod{\lambda}$, which is
unambiguously defined if it is irreducible. Likewise for $\rho_g$
and $\rhobar_g$.
\begin{lem}\label{cp}
\begin{enumerate}
\item Suppose that $\rhobar_f$ and $\rhobar_g$ are irreducible, that $\ell>k'$ and $\ell\nmid N$.
Suppose (for some $p\mid\mid N$) that there is no normalised newform
$h$ of level dividing $N/p$ and trivial character, of weight $k'$
with $a_q(h)\equiv a_q(f)\pmod{\lambda}$ for all primes $q\nmid \ell
N$, or of weight $k$ with $a_q(h)\equiv a_q(g)\pmod{\lambda}$ for
all primes $q\nmid \ell N$. Then the $\lambda$ part of $c_p(t)$ is
trivial (for any $t$).
\item If $\lambda\mid \ell$ with $\ell\nmid N$ and $\ell >k'+k-1$ then the $\lambda$ part of
$c_{\ell}(t)$ is trivial (for any $t$).
\end{enumerate}
\end{lem}
\begin{proof}
\begin{enumerate}
\item Applying a level-lowering theorem (Theorem 1.1 of \cite{Di}, see also \cite{R2,R3}), $\rhobar_f$ and $\rhobar_g$
are both ramified at $p$. However, since $p\mid\mid N$, the action
of $I_p$ on each of $M_{f,\lambda}$ and $M_{g,\lambda}$ is
unipotent, by Theorem 7.5 of \cite{L}, as recalled in Theorem
4.2.7(3)(b) of \cite{Hi2}, for a convenient reference. It follows
that both $\rhobar_f\otimes \rhobar_g$ and $\rho_f\otimes\rho_g$
have $I_p$-fixed subspace of dimension precisely $2$, hence that
$A_{\lambda}^{I_p}$ is divisible. As noted above, this implies that
the $\lambda$-part of $c_p(t)$ is trivial.
\item It follows from Lemma~5.7 of \cite{DFG} (whose proof relies on an
application, at the end of Section~2.2, of the results of \cite{Fa})
that $\MMM_{f\otimes g,\lambda}$ is the
$O_{\lambda}[\Gal(\Qbar_{\ell}/\QQ_{\ell})]$-module associated to
the filtered $\phi$-module $\MMM_{f\otimes g,\dR}\otimes
O_{\lambda}$ (identified with the crystalline realisation) by the
functor they call $\mathbb{V}$. (This property is part of the
definition of an $S$-integral premotivic structure given in
Section~1.2 of \cite{DFG}.) Given this, the lemma follows from
Theorem 4.1(iii) of \cite{BK}. (That $\mathbb{V}$ is the same as the
functor used in Theorem~4.1 of \cite{BK} follows from the first
paragraph of 2(h) of \cite{Fa}.)
\end{enumerate}
\end{proof}
\begin{cor}\label{pred} Suppose that $N$ is square-free.
Assume the conditions of Lemma \ref{cp}(1), for all primes $p\mid
N$, and of Lemma \ref{cp}(2), and also that (for some $k\leq t\leq
k'-1$)
$$\ord_{\lambda}\left(\frac{L_{f\otimes g}(t)}{\pi^{2t-(k-1)}\langle f,f\rangle}\right)>0.$$
 Then the Bloch-Kato conjecture predicts that
$\ord_{\lambda}(\#\Sha(t))>0$, so predicts that the Selmer group
$H^1_f(\QQ,A_{\lambda}(t))$ is non-trivial.
\end{cor}
The goal of this paper is to construct (under further hypotheses) a
non-zero element of $H^1_f(\QQ,A_{\lambda}(t))$, in the case that
$t$ is the near-central point $t=\frac{k'+k-2}{2}$.
\begin{lem}\label{torsion}
If $\ell\nmid N$, $\ell>k'-1$ and $k<t<k'-1$ then the
$\lambda$-parts of $\#H^0(\QQ,A(t))$ and $\#H^0(\QQ,\check{A}(1-t))$
are trivial.
\end{lem}
\begin{proof} If not, then either $A[\lambda](t)$ or $\check{A}[\lambda](1-t)$
would have a trivial composition factor. The composition factors of
$\rhobar_f|_{I_{\ell}}$ are either $\chi^0,\chi^{1-k}$ (in the
ordinary case, with $\chi$ the cyclotomic character) or
$\psi^{1-k},\psi^{\ell (1-k)}$ (in the non-ordinary case, with
$\psi$ a fundamental character of level $2$). This follows from
theorems of Deligne and Fontaine, which are Theorems 2.5 and 2.6 of
\cite{Ed}. Noting that $\psi$ has order $\ell^2-1$, with
$\psi^{\ell+1}=\chi$, the composition factors of
$(\rhobar_f\otimes\rhobar_g)|_{I_{\ell}}$ are of the form
$\psi^a,\psi^b,\psi^c,\psi^d$, with $1-\ell^2<a,b,c,d\leq 0$ and
each of $a,b,c,d$ congruent to either $0,1-k,1-k'$ or
$2-k-k'\pmod{\ell}$. Twisting by $t$ is the same as multiplying by
$\psi^{(\ell+1)t}$. This exponent is congruent to $t\pmod{\ell}$,
and $k<t<k'-1$. Adding to this the possible values for
$a,b,c,d\pmod{\ell}$ can never produce $0$ or $1$. Hence neither
$A[\lambda](t)$ nor $\hat{A}[\lambda](1-t)$ can have a trivial
composition factor (even when restricted to $I_{\ell}$).
\end{proof}
\section{A $4$-dimensional Galois representation}
Let $f,g$ be as in \S\S 2,3, both of exact level $N>1$. Let
$\lambda\mid\ell$ be a divisor of $\frac{L_{f\otimes
g}(t)}{\pi^{2t-(k-1)}\langle f,f\rangle}$, with $\ell\nmid N(k')!$
and $t=(k'+k-2)/2$. Now suppose that $f$ and $g$ have the same
Atkin-Lehner eigenvalues for each $p\mid N$, and let $F_{f,g}$ be
some genus-$2$ Yoshida lift associated with a factorisation
$N=N_1N_2$, as in \S 8 below. (It is of type
$\Sym^j\otimes\det^{\kappa}$, with $j=k-2, \kappa=2+\frac{k'-k}{2}$.
Note that $j+2\kappa-3=k'-1$.)

Suppose that there is a cusp form $G$ for $\Gamma_0^{(2)}(N)$, an
eigenvector for all the local Hecke algebras at $p\nmid N$, not
itself a Yoshida lift of the same $f$ and $g$, such that there is a
congruence $\pmod{\lambda}$ of all Hecke eigenvalues (for $p\nmid
N$) between $G$ and $F_{f,g}$. In particular, if $\mu_G(p)$ is the
eigenvalue for $T(p)$ on $G$ (defined as in \S 2.1 of \cite{Ar},
replacing $\Sp_4(\ZZ)$ by $\Gamma_0^{(2)}(N)$), then
\begin{equation}\label{cong}
\mu_G(p)\equiv a_p(f)+p^{(k'-k)/2}a_p(g)\pmod{\lambda},\,\,\text{
for all }p\nmid N.
\end{equation}
Under certain additional hypotheses, we prove in \S 9 below, the
existence of such a $G$. (We enlarge $K$ if necessary, to contain
the Hecke eigenvalues of $G$.)

Let $\Pi_G$ be an automorphic representation of $\GSp_4(\AAA)$
associated to $G$ as in 3.2 of \cite{Sc} and 3.5 of \cite{AS}. (This
$\Pi_G$ is not necessarily uniquely determined by $G$, but its local
components at $p\nmid N$ are.) By Theorem I of \cite{We}, there is
an associated continuous, linear representation
$$\rho_G: \Gal(\Qbar/\QQ)\rightarrow\GL_4(\Qbar_{\ell}).$$
By enlarging $K$ if necessary, we may assume that it takes values in
$\GL_4(K_{\lambda})$.
\begin{lem}\label{endocap}
Suppose that there exists a $G$ as above. Suppose also that
$\lambda$ is not a congruence prime for $f$ in $S_{k'}(\Gamma_0(N))$
or $g$ in $S_k(\Gamma_0(N))$, that $\ell>k'$, and that $\rhobar_f$
and $\rhobar_g$ are irreducible representations of
$\Gal(\Qbar/\QQ)$.
\begin{enumerate}
\item $\Pi_G$ is not a weak endoscopic lift.
\item $\Pi_G$ is not CAP.
\end{enumerate}
\end{lem}
By $\lambda$ not being a congruence prime for $f$ in
$S_{k'}(\Gamma_0(N))$, we mean that there does not exist a different
Hecke eigenform $h\in S_{k'}(\Gamma_0(N))$, and a prime $\lambda'$
dividing $\lambda$ in a sufficiently large extension, such that
$a_p(h)\equiv a_p(f)\pmod{\lambda'}$ for all primes $p\nmid \ell N$.
\begin{proof}
\begin{enumerate}
\item If $\Pi_G$ were a weak endoscopic lift then there would have
to exist newforms $f'\in S_{k'}(\Gamma_0(N)), h\in
S_{k}(\Gamma_0(N))$ such that $\mu_G(p)=a_p(f')+p^{(k'-k)/2}a_p(h)$
for almost all primes $p$. (See the introduction of \cite{We} for a
precise definition of weak endoscopic lift, and (3) of Hypothesis A
of \cite{We} for this consequence.) We have then
$$a_p(f')+p^{(k'-k)/2}a_p(h)\equiv
a_p(f)+p^{(k'-k)/2}a_p(g)\pmod{\lambda},$$ for almost all primes
$p$. Consequently, using $\ell>4$ and the Brauer-Nesbitt theorem,
$$\rhobar_f\oplus\rhobar_g\left((k-k')/2\right)\simeq
\rhobar_{f'}\oplus\rhobar_h((k-k')/2).$$ Now $\rhobar_{f'}$ could
not be isomorphic to $\rhobar_g\left((k-k')/2\right)$, since the
restrictions to $I_{\ell}$ give different characters (using
$\ell>k'$). The only way to reconcile the two sides of the above
isomorphism is for $\rhobar_f\simeq \rhobar_{f'}$. Given that
$\lambda$ is not a congruence prime for $f$ in
$S_{k'}(\Gamma_0(N))$, we must have $f'=f$, and similarly $h=g$. It
follows from (4) and (6) of Hypothesis A of \cite{We} that $\Pi_G$
must be associated to some Yoshida lift $F'_{f,g}$ of $f$ and $g$.
(Those $p\mid N$ for which the local component is $\Pi_v^+$ rather
than $\Pi_v^-$ are the divisors of $N_1$.) By (6) of Hypothesis A of
\cite{We}, the multiplicity of $\Pi_G$ in the discrete spectrum is
one. By Lemmes 1.2.8 and 1.2.10 of \cite{SU}, the local
representation $\Pi_p$ of $\GSp(4,\QQ_p)$, for $p\mid N$, is that
labelled VIa in \cite{Sc}. By Table 3 of \cite{Sc}, the spaces of
$\Gamma_0^{(2)}(\ZZ_p)$-fixed vectors in $\Pi_p$ are
$1$-dimensional. It follows that (up to scaling), $G=F'_{f,g}$,
contrary to hypothesis.
\item By Corollary 4.5 of \cite{PS}, $\Pi_G$ could only be CAP for
a Siegel parabolic subgroup, but then, as on p.74 of \cite{We}, we
would have $k=2$ and
$$\mu_G(p)=a_p(f')+\chi(p)p^{k'/2}+\chi(p)p^{(k'/2)-1},$$
for some newform $f'\in S_{k'}(\Gamma_0(N))$ and $\chi$ a quadratic
or trivial character. This is incompatible with $\mu_G(p)\equiv
a_p(f)+p^{(k'-k)/2}a_p(g)\pmod{\lambda}$ and the irreducibility of
$\rhobar_f$ and $\rhobar_g$.
\end{enumerate}
\end{proof}
Note that the proof of Hypothesis A (on which Theorem I also
depends) is not in \cite{We}, but has now appeared in \cite{We2}.
\begin{lem}\label{irred}
Let $G$ be as in Lemma \ref{endocap}. Then the representation
$\rho_G$ is irreducible.
\end{lem}
\begin{proof} Suppose that $\rho_G$ is reducible. It cannot have any
$1$-dimensional composition factor, since $\rhobar_G$ has
$2$-dimensional irreducible composition factors $\rhobar_f$ and
$\rhobar_g((k-k')/2)$. (The factors are well-defined, even though
$\rhobar_G$ isn't.) Looking at the list, in 3.2.6 of \cite{SU}, of
possibilities for the composition factors of $\rho_G$, we must be in
Cas B, (iv) or (v). But as in 3.2.6 of \cite{SU}, $\Pi_G$ would be
CAP in one case, a weak endoscopic lift in the other, and both of
these are ruled out by Lemma \ref{endocap}.
\end{proof}
Let $V$, a $4$-dimensional vector space over $K_{\lambda}$, be the
space of the representation $\rho_G$. Choose a
$\Gal(\Qbar/\QQ)$-invariant $O_{\lambda}$-lattice $T$ in $V$, and
let $W:=V/T$. Let $\rhobar_G$ be the representation of
$\Gal(\Qbar/\QQ)$ on $W[\lambda]\simeq T/\lambda T$. This depends on
the choice of $T$, but we may choose $T$ in such a way that
$\rhobar_G$ has $\rhobar_g((k-k')/2)$ as a submodule and $\rhobar_f$
as a quotient. Assume that this has been done.
\begin{lem}\label{notsum}
$T$ may be chosen in such a way that furthermore $\rhobar_f$ is not
a submodule of $\rhobar_G$, i.e. so that the extension of
$\rhobar_f$ by $\rhobar_g((k-k')/2)$ is not split.
\end{lem}
\begin{proof} We argue as in the proof of Proposition 2.1 of
\cite{R}. Choose an $O_{\lambda}$-basis for $T$, so that
$\rho_G(\Gal(\Qbar/\QQ))\subset \GL_4(O_{\lambda})$. Assuming the
lemma is false, we prove by induction that for all $i\geq 1$ there
exists $M_i=\begin{pmatrix} I_2&S_i\\0_2&I_2\end{pmatrix}\in
\GL_4(O_{\lambda})$ such that $M_i\,\rho_G(\Gal(\Qbar/\QQ))M_i^{-1}$
consists of matrices of the form
$\begin{pmatrix}A&\lambda^iB\\\lambda C&D\end{pmatrix}$, with
$A,B,C,D\in M_2(O_{\lambda})$, and with $S_{i}\equiv
S_{i-1}\pmod{\lambda^{i-1}}$. Then letting $S=\lim S_i$ and
$M=\begin{pmatrix} I_2&S\\0_2&I_2\end{pmatrix}$,
$M\,\rho_G(\Gal(\Qbar/\QQ))M^{-1}$ consists of matrices of the form
$\begin{pmatrix}A&0_2\\\lambda C&D\end{pmatrix}$, contradicting the
irreducibility of $\rho_G$.

By assumption, $\rhobar_f$ is a submodule of $\rhobar_G$ (i.e.
$\rhobar_G$ is semi-simple), so we have $M_1$. This is the base
step. Now suppose that we have $M_i$. We must try to produce
$M_{i+1}$. Let $P=\begin{pmatrix}I_2&0_2\\0_2&\lambda
I_2\end{pmatrix}$. Then
$P^iM_i\,\rho_G(\Gal(\Qbar/\QQ))M_i^{-1}P^{-i}$ consists of matrices
of the form $\begin{pmatrix}A&B\\\lambda^{i+1} C&D\end{pmatrix}$.
Now let $U$ be a matrix of the form
$\begin{pmatrix}I_2&B'\\0_2&I_2\end{pmatrix}$ such that
$UP^iM_i\,\rho_G(\Gal(\Qbar/\QQ))M_i^{-1}P^{-i}U^{-1}$ consists of
matrices of the form $\begin{pmatrix}\tilde{A}&\lambda
\tilde{B}\\\lambda^{i+1}\tilde{C}&\tilde{D}\end{pmatrix}$. This
exists because we are assuming that not only $\rhobar_G$, but any
other reduction with submodule $\rhobar_g((k-k')/2)$, is
semi-simple. Now just let $M_{i+1}=P^{-i}UP^iM_i$. Note that since
$P^{-i}UP^i=\begin{pmatrix}I_2&\lambda^iB'\\0_2&I_2\end{pmatrix}$,
it is clear that $M_{i+1}$ is of the form $\begin{pmatrix}
I_2&S_{i+1}\\0_2&I_2\end{pmatrix}$, with $S_{i+1}\equiv
S_i\pmod{\lambda^i}$.
\end{proof}
We remark that, though the first $T$ chosen may give semi-simple
$\rhobar_G$, the lemma shows that there will be another choice that
gives a non-trivial extension. Compare with the situation for
$5$-torsion on elliptic curves in the isogeny class of conductor
$11$.
\section{A non-zero element in a Bloch-Kato Selmer group}
Let $G$ be as in the previous section. Then by Lemma \ref{notsum},
$\rhobar_G$ is a non-trivial extension of $\rhobar_f$ by
$\rhobar_g((k-k')/2)$:
$$\begin{CD} 0@>>>\rhobar_g((k-k')/2)@>>>\rhobar_G
@>>>\rhobar_f@>>>0.\end{CD}$$ Applying
$\Hom_{\FF_{\lambda}}(\rhobar_f,\,\_\_)$ to the exact sequence, and
pulling back the inclusion of the trivial module in
$\Hom_{\FF_{\lambda}}(\rhobar_f,\rhobar_f)$, we get a non-trivial
extension of the trivial module by
$\Hom(\rhobar_f,\rhobar_g((k-k')/2)$. Thus we get a non-zero class
in $H^1(\QQ,\Hom_{\FF_{\lambda}}(\rhobar_f,\rhobar_g((k-k')/2)))$,
in the standard way. (Lifting the identity to a section $s\in
\Hom_{\FF_{\lambda}}(\rhobar_f,\rhobar_G)$, a representing cocycle
is $g\mapsto g.s-s$, where $(g.s)(x)=g(s(g^{-1}(x)))$.)

Now the dual of $\rhobar_f$ is $\rhobar_f(k'-1)$, so
$$\Hom_{\FF_{\lambda}}(\rhobar_f,\rhobar_g((k-k')/2)))\simeq
\rhobar_f(k'-1)\otimes\rhobar_g((k-k')/2)\simeq
\rhobar_f\otimes\rhobar_g((k'+k-2)/2).$$ In the notation of \S 3,
this is $ A[\lambda]((k'+k-2)/2)$. So we have a non-zero class $c\in
H^1(\QQ,A[\lambda]((k'+k-2)/2))$. By Lemma \ref{torsion},
$H^0(\QQ,A_{\lambda}((k'+k-2)/2))$ is trivial, so we get a non-zero
class $d\in H^1(\QQ,A_{\lambda}((k'+k-2)/2))$, the image of $c$
under the map induced by inclusion.

\begin{prop}\label{selmer}
Let $f\in S_{k'}(\Gamma_0(N)), g\in S_k(\Gamma_0(N))$ be normalised
newforms of square-free level $N>1$, with $k'>k\geq 2$. Suppose that
at each prime $p\mid N$, $f$ and $g$ share the eigenvalue of the
Atkin-Lehner involution. Let $\lambda\mid\ell$ be a divisor of
$\frac{L_{f\otimes g}((k'+k-2)/2)}{\pi^{k'-1}\langle f,f\rangle}$,
with $\ell\nmid N$ and $\ell>\frac{3k'+k-2}{2}$. Suppose also that
$\lambda$ is not a congruence prime for $f$ in $S_{k'}(\Gamma_0(N))$
or $g$ in $S_k(\Gamma_0(N))$, and that $\rhobar_f$ and $\rhobar_g$
are irreducible representations of $\Gal(\Qbar/\QQ)$. Assume, for
each $p\mid N$, the conditions of Lemma \ref{cp}(1). Finally,
suppose that there exists $G\in S_{\rho}(\Gamma_0^{(2)}(N))$ as in
the second paragraph of \S 4. Then the Bloch-Kato Selmer group
$H^1_f(\QQ,A_{\lambda}((k'+k-2)/2))$ is non-zero.
\end{prop}
\begin{remar} Note that Corollary \ref{Nprime} gives sufficient
conditions for the existence of $G$.
\end{remar}
\begin{proof} We will show that the non-zero element $d\in H^1(\QQ,A_{\lambda}((k'+k-2)/2))$
satisfies $\res_p(d)\in H^1_f(\QQ_p, A_{\lambda}((k'+k-2)/2))$ for
each prime $p$.
\begin{enumerate}
\item If $p\nmid \ell N$ then $\rho_G|_{I_p}$ is trivial, so
certainly $$\begin{CD}
0@>>>\rhobar_g((k-k')/2)|_{I_p}@>>>\rhobar_G|_{I_p}
@>>>\rhobar_f|_{I_p}@>>>0\end{CD}$$ splits, showing that
$\res_p(c)\in \ker(H^1(\QQ_p,A[\lambda]((k'+k-2)/2))\rightarrow
H^1(I_p,A[\lambda]((k'+k-2)/2)))$, hence that $\res_p(d)\in
\ker(H^1(\QQ_p,A_{\lambda}((k'+k-2)/2))\rightarrow
H^1(I_p,A_{\lambda}((k'+k-2)/2)))$. Since $A_{\lambda}^{I_p}$ is
divisible (in this case the whole of $A_{\lambda}$), this shows that
$\res_p(d)\in H^1_f(\QQ_p,A_{\lambda}((k'+k-2)/2))$, as in Lemma 7.4
of \cite{Br}.
\item If $p=\ell$ then we may prove $\res_p(d)\in H^1_f(\QQ_p, A_{\lambda}((k'+k-2)/2))$
just as in Lemma 7.2 of \cite{Du}. Since $\ell\nmid N$,
$\rho_G|_{D_{\ell}}$ is crystalline; see Theorem 3.2(ii) of
\cite{U}, which refers to \cite{Fa} and \cite{CF}. It is for this
case that we need the condition $\ell>\frac{3k'+k-2}{2}$. This
$\frac{3k'+k-2}{2}$ arises as the span of the ``weights''
$\{1-k',0\}$ of $\rhobar_f^*$ and $\{(k'-k)/2, (k'+k-2)/2\}$ of
$\rhobar_g((k-k')/2)$. See the proof of Lemma 7.2 of \cite{Du} for
comparison.
\item Now consider the case that $p\mid N$. As in the proof of Lemma
\ref{cp}(1), the action of $I_p$ on
$\MMM_{f,\lambda}/\lambda\MMM_{f,\lambda}$ and
$\MMM_{g,\lambda}/\lambda\MMM_{g,\lambda}$ is non-trivial and
unipotent. Hence we may choose a basis for $W[\lambda]$ (notation as
in the previous section) such that for any $\sigma\in I_p$,
$\rhobar_G(\sigma)$ is represented by
$\exp(t_{\ell}(\sigma)\tilde{N})$, with $t_{\ell}:I_p\rightarrow
\ZZ_{\ell}(1)$ the standard tamely ramified character and
$\tilde{N}$ of the form
$\tilde{N}=\begin{pmatrix}A&B\\0_2&A\end{pmatrix}$, with
$A=\begin{pmatrix}0&1\\0&0\end{pmatrix}$. (Note that $A$ plays the
r\^ole of $\tilde{N}$ for the $2$-dimensional representations
$\rhobar_f|_{I_p}$ and $\rhobar_g|_{I_p}$.) By Theorem 2.2.5(1) of
\cite{GT}, $\tilde{N}^2=0$. To see that the conditions of that
theorem are satisfied here, firstly $\rho_G$ is irreducible by Lemma
\ref{irred}, secondly $\rho_G$ is symplectic by Theorem 2 of
\cite{We4}. Lastly, given that the local component $\Pi_p$ of
$\Pi_G$ has a non-zero vector fixed by $\Gamma_0^{(2)}(\ZZ_p)$ but
none fixed by $\GSp_4(\ZZ_p)$, an inspection of Table 3 in \cite{Sc}
reveals that it is always the case that either the subspace of
$\Pi_p$ fixed by the Siegel parahoric $\Gamma_0^{(2)}(\ZZ_p)$, or
that fixed by a Klingen parahoric, is $1$-dimensional. (Note that if
$\Pi_p$ had a non-zero vector fixed by $\GSp_4(\ZZ_p)$ then, by
Theorem I of \cite{We}, $\rho_G$ would be unramified at $p$,
contrary to $\rhobar_G$ having $\rhobar_f$ as a quotient.)

Since $\tilde{N}^2=0$, $B$ must be of the form $B=\begin{pmatrix}
0&b\\0&0\end{pmatrix}$. Writing elements of
$\Hom_{\FF_{\lambda}}(\rhobar_f,\rhobar_g((k-k')/2))$ as $2$-by-$2$
matrices in the obvious way, a short calculation shows that
$c|_{I_p}$ is represented by the cocycle
$\sigma\mapsto\begin{pmatrix}0&t_{\ell}(\sigma)b\\0&0\end{pmatrix}$,
which is the coboundary $\sigma\mapsto
\sigma\left(\begin{pmatrix}0&0\\0&b\end{pmatrix}\right)-\begin{pmatrix}0&0\\0&b\end{pmatrix}$.
Since $c|_{I_p}=0$, $d|_{I_p}=0$. As already noted in the proof of
Lemma \ref{cp}, $A_{\lambda}^{I_p}$ is divisible, so we may deduce
as in (1) that $\res_p(d)\in H^1_f(\QQ_p, A_{\lambda}((k'+k-2)/2))$.
\end{enumerate}
\end{proof}
\begin{remar} We could have used a different formulation of the
Bloch-Kato conjecture, for the incomplete $L$-function with Euler
factors at $p\mid N$ missing, as in (59) of \cite{DFG}, similarly
using the exact sequence in their Lemma 2.1. This would have
involved a Selmer group with no local restrictions at $p\mid N$, and
eliminated the Tamagawa factors at $p\mid N$. Hence we could have
avoided the related difficulties of showing triviality of
$\lambda$-parts of Tamagawa factors (at $p\mid N$ but not at
$p=\ell$) and proving that local conditions at $p\mid N$ are
satisfied. However, we chose to assume a little more than necessary
(i.e. the conditions of Lemma \ref{cp}(1)), then use it to prove
something a bit stronger.
\end{remar}

\section{The doubling method with differential operators}

We mainly recall some properties of the doubling method in the
setting of holomorphic Siegel modular forms (with invariant
differential operators). As long as one does not insist on explicit
constants and explicit $\Gamma$-factors, everything works more
generally for arbitrary polynomial representations as automorphy
factors, see \cite[section 2]{BS3}, \cite{I}.

\subsection{Construction of holomorphic differential operators}

We construct holomorphic differential operators on $\HH_{2n}$ with
certain equivariance properties. We combine the constructions from
\cite{BoeFJII} and \cite{BSY}; a similar strategy was also used by
\cite{Koz}.
\\
We decompose $Z\in \HH_{2n}$ as
$$ Z=(z_{ij})=\left(\begin{array}{cc} z_1 & z_2\\
z_2^t & z_4\end{array}\right)\qquad (z_1,z_4\in \HH_n).$$ We also
use the natural embedding $Sp(n)\times Sp(n)\hookrightarrow Sp(2n)$,
defined by
$$(M_1,M_2)\mapsto M_1^{\uparrow}\cdot M_2^{\downarrow}:=
\left(\begin{array}{cccc}
A_1 & 0   & B_1 & 0\\
0   & A_2 & 0   & B_2\\
C_1 & 0   & D_1 & 0\\
0   & C_2 & 0   & D_2 \end{array}\right),\qquad M_i=\left(\begin{array}{cc}
A_i & B_i\\
C_i & D_i\end{array}\right)\in Sp(n).$$
The differential operator matrix ${ \partial}=
(\partial_{ij})$ with $\partial_{ij}=
\frac{1+\delta_{ij}}{2}\frac{\partial}{\partial z_{ij}}$ will then be
decomposed in block matrices of size $n$, denoted by
$$\partial=\left(\begin{array}{cc} \partial_1 & \partial_2\\
\partial_2^t & \partial_4\end{array}\right).$$
We realize the symmetric tensor representation $\sigma_{\nu}:=
Sym^{\nu} $ of $GL(n,{\mathbb C})$ in the usual way on the space
$V_{\nu}:= {\mathbb C}[X_1,\dots X_n]_{\nu} $ (of homogeneous
polynomials of degree $\nu$). For $V_{\nu} $ -valued functions $f$
on $\HH_n$, $\alpha,\beta\in {\mathbb C}$ and $M\in Sp(n,{\mathbb
R})$ we define the slash-operator by
$$(f\mid_{\alpha,\beta, \sigma_{\nu}}M)(z):=
\det(cz+d)^{-\alpha}\det(c\bar{z}+d)^{-\beta}
\sigma_{\nu}(cz+d)^{-1} f(M\langle z\rangle).$$
We may ignore the
ambiguity of the powers $\alpha,\beta\in {\mathbb C}$ most of the
time. If $\beta=0$ or $\nu=0$ we just omit them from the slash
operator.

\begin{prop} \label{definition_diffop_prop}
For nonnegative integers $\mu,\nu$ there is a (nonzero) holomorphic
differential operator ${\mathbb D}_{\alpha}(\mu,\nu)$ mapping
scalar-valued $C^{\infty}$ functions $F$ on $\HH_{2n}$ to
$V_{\nu}\otimes V_{\nu}$-valued functions on $\HH_n\times \HH_n$,
satisfying

\begin{equation}{\mathbb D}_{\alpha}(\mu,\nu)
(F\mid_{\alpha,\beta}(M_1^{\uparrow}M_2^{\downarrow}))
=\left({\mathbb D}_{\alpha}(\mu,\nu)(F)\right)
\mid^{z_1}_{\alpha+\mu,\beta,\sigma_{\nu}}M_1
\mid^{z_4}_{\alpha+\mu,\beta,\sigma_{\nu}}M_2
\label{weights}\end{equation} for all $M_1,M_2\in Sp(n,{\mathbb
R})$; the upper index at the slash operator
indicates, for which variables $M_i$ is applied.\\
More precisely, there is a $V_{\nu}\otimes V_{\nu}$ -valued nonzero
polynomial $Q(\alpha,{\bf T})= Q^{(\mu,\nu)}_{\alpha}({\bf T})$  in
the variables $\alpha$ and ${\bf T}$ (where ${\bf T}$ is a symmetric
$2n\times 2n$ matrix of variables), with rational coefficients, such
that
$${\mathbb D}_{\alpha}(\mu,\nu)= Q^{(\mu,\nu)}_{\alpha}
(\partial_{ij})\mid_{z_2=0}.$$
 The differential operator ${\mathbb D}_{\alpha}(\mu,\nu)$
has the additional symmetry property
$${\mathbb D}_{\alpha}(\mu,\nu)(F\mid V)=
{\mathbb D}_{\alpha}(\mu,\nu)(F)^{\star },$$
where $V$ is the operator defined on functions on $\HH_{2n}$ by
$$F\longmapsto (F\mid V)\,\left(\left(\begin{array}{cc} z_1 & z_2\\
z_2^t & z_4\end{array}\right) \right)= F\left( \left(\begin{array}{cc} z_4 & z_2^t\\
z_2 & z_1\end{array}\right) \right)$$ and for a function $g$ on
$\HH_n\times \HH_n$ we put $g^{\star}(z,w):=g(w,z)$.

\end{prop}

\begin{remar} We allow arbitrary ``complex weights'' $\alpha$ here; note that
there is no ambiguity in this as long as we use the same branch
of $\log \det(CZ+D)$ to define the $\det(CZ+D)^s$ on both sides of
(\ref{weights}). \\
Note also that the differential operators do
not depend at all on $\beta$.
\end{remar}
\begin{proof}
We recall from \cite[Satz 2]{BoeFJII} the existence of an explicitly given
differential operator

$${\mathcal D}_{\alpha}= (-1)^n C_n\left(\alpha-n+\frac{1}{2}\right)\det(\partial_2)+...
+ \det(z_2)\cdot \det(\partial_{ij})$$
with
\begin{equation}\label{Cn_equation}C_n(s):= s\left(s+\frac{1}{2}\right)\dots \left(s+\frac{n-1}{2}\right)=
\frac{\Gamma_n(s+\frac{n+1}{2})}{\Gamma_n(s+\frac{n-1}{2})} \qquad
\left(\Gamma_n(s)=\pi^{\frac{n(n-1)}{4}}\prod_{j=0}^{n-1}\Gamma\left(s-\frac{j}{2}\right)
\right).
\end{equation}
 This operator is compatible with the action of $
Sp(n,{\mathbb R})\times Sp(n,{\mathbb R}) \hookrightarrow
Sp(2n,{\mathbb R}),$ increasing the weight $\alpha$ by one (without
restriction!), i.e.
$${\mathcal D}_{\alpha} (F\mid_{\alpha,\beta} M^{\uparrow}_1\cdot M^{\downarrow}_2)
=
({\mathcal D}_{\alpha}F)\mid_{\alpha+1,\beta} M^{\uparrow}_1\cdot M^{\downarrow}_2
,\qquad (M_i\in Sp(n,{\mathbb R}))$$
We put
$${\mathcal D}_{\alpha}^{\mu}:= {\mathcal D}_{\alpha+\mu-1} \circ\dots
\circ {\mathcal D}_{\alpha}.$$

\begin{remar} The combinatorics of this operator is not known explicitly for
general $\mu$.
\end{remar}
The second type of differential operators maps scalar-valued
functions on $\HH_{2n}$ to ${\mathbb C}[X_1,\dots ,
X_n]_{\nu}\otimes {\mathbb C}[Y_{1},\dots , Y_{n}]_{\nu}$-valued
functions on $\HH_n\times \HH_n$, changing the automorphy factor from
$\det^{\alpha}$ on $GL(2n,{\mathbb C})$ to
$(\det^{\alpha}\otimes Sym^{\nu})\boxtimes (\det^{\alpha}\otimes Sym^{\nu})$ on
$GL(n,\mathbb C)\times GL(n,{\mathbb C)}$. This operator
was introduced in \cite[section 2]{BSY}; it is a
special feature that we know the
combinatorics in this case quite explicitly:

\begin{equation}L^{\nu}_{\alpha}:=
\frac{1}{(2\pi i)^{\nu}{\alpha}^{[\nu]}}\left(\sum_{0\leq 2j\leq \nu}
\frac{1}{j!(\nu-2j)! (2-\alpha-\nu)^{[j]}}
(D_{\uparrow}D_{\downarrow})^j(D-D_{\uparrow}-D_{\downarrow})^{\nu-2j}
\right)_{z_2=0};
\label{satoh}
\end{equation}
here we use the same notation as in \cite{BSY}:
\begin{eqnarray*}
\alpha^{[j]} &=& \alpha(\alpha+1)\dots (\alpha+j-1)=
\frac{\Gamma(\alpha+j)}{\Gamma(\alpha)}\\
D&=& \partial[(X_1,\dots, X_n,Y_1,\dots,Y_n)^t]\\
D_{\uparrow}&=&\partial[(X_1,\dots,X_n,0,\dots , 0)^t]\\
D_{\downarrow}&=&\partial[(0,\dots ,0;Y_1,\dots ,Y_n)^t],
\end{eqnarray*}
where
$A[x]:=x^t\,Ax$; we remark that
$$
D-D_{\uparrow}-D_{\downarrow} = (X_1,\dots X_n;0,\dots 0)\cdot \partial_2
\cdot (0,\dots ,0; Y_1,\dots Y_n)^t.$$
In \cite{BSY} the weight was a natural number $k$, but everything
works also for arbitrary complex $\alpha$ instead. (Due to the
normalization of \cite{BSY}, we have to omit certain finitely many
$\alpha$.)
\\
We put
$${\mathbb D}_{\alpha}(\mu,\nu):= L^{\nu}_{\alpha+\mu}
\circ {\mathcal D}_{\alpha}^{\mu}.$$
This operator has all the requested properties, except for
the fact that the coefficients are not polynomials in $\alpha$ but
rational functions.
\end{proof}

\subsection{Some combinatorics}

Then we consider the
function $h_{\alpha,\beta}$ defined on ${\mathbb H}_{2n}$ by
$$h_{\alpha,\beta}(Z):= \det(z_1+z_2+z_2^t+z_4)^{-\alpha}
\overline{\det(z_1+z_2+z_2^t+z_4)}^{-\beta}$$
and we note that
(following \cite[(1.25)]{BCG})
$$ {\mathcal D}_{\alpha}^{\mu} h_{\alpha,\beta}
= A_{\alpha,\mu}\cdot h_{\alpha+\mu,\beta}$$
with
$$A_{\alpha,\mu}=
\frac{\Gamma_n(\alpha+\mu)}{\Gamma_n(\alpha)}
\frac{\Gamma_n(\alpha+\mu-\frac{n}{2})}{\Gamma_n({\alpha}-\frac{n}{2})}
$$
and also
$$L_{\alpha}^{\nu}h_{\alpha,\beta}=B_{\alpha,\nu}
\sigma_{\nu}(z_1+z_4)^{-1} \left(\sum
X_iY_i\right)^{\nu}\det(z_1+z_4)^{-\alpha}
\overline{\det(z_1+z_4)}^{-\beta}$$ with

$$B_{\alpha,\nu} = \frac{1}{(-2\pi i)^{\nu}\nu!}
\frac{\Gamma(2\alpha-2+\nu)}{\Gamma(2\alpha-2)}
\frac{\Gamma(\alpha-1)}{\Gamma(\alpha+\nu-1)},$$ following
\cite[Lemma 4.2]{BSY}.
\\[0.6cm]
For later purposes we
summarize here some additional properties of these differential operators: \\
First we note that
${\mathbb D}_{\alpha}(\mu,\nu)$ is a homogeneous polynomial (of
degree $n\mu+\nu$) in the partial derivatives; we decompose it as

$$  {\mathbb D}_{\alpha}(\mu,\nu)=
{\mathcal M}+{\mathcal R},$$
where the
``main term `` ${\mathcal M}$
denotes the part free of derivatives w.r.t. $z_1$ or $z_4$.

\begin{lem}\label{combinatoricslemma}
a) All the monomials occurring in the ``remainder term''
${\mathcal R}$ have positive degree in the
partial derivatives w.r.t. $z_1$ \underline{and} $z_4$.\\
b) The ``main term''  ${\mathcal M}$ is of the form
$${\mathcal M}=C_{\alpha}(\mu,\nu) \left(D-D^{\uparrow} -
D^{\downarrow}\right)^{\nu} \det(\partial_2)^{\mu}$$
with
$$C_{\alpha}(\mu,\nu)=\frac{1}{(\alpha+\mu)^{[\nu]}\nu!}
\prod_{j=0}^{\mu-1} C_n \left(\alpha-n+
\frac{\mu+\nu'+j}{2}\right)\qquad \left(\nu':=\frac{\nu}{n}\right),
$$ where $C_n(s)$ is as in equation (\ref{Cn_equation}).\\
c) For the polynomial $Q_{\alpha}^{\mu,\nu}({\mathbf T})$ with the symmetric
 matrix ${\mathbf T}=\left(\begin{array}{cc} {\mathbf T}_1 & {\mathbf T}_2\\
{\mathbf T}_2^t & {\mathbf T}_4\end{array}\right)$ of size $2n$
this means
\begin{equation}
Q_{\alpha}^{\mu,\nu}({\mathbf T})= C_{\alpha}(\mu,\nu)
\left(2(X_1,\dots ,X_n){\mathbf T}_2(Y_1,\dots ,Y_n)^t\right)^{\nu}
\det({\mathbf T}_2)^{\mu} +(*) \label{nose}\, ,\end{equation}
where (*) contains only contributions with positive degree in ${\mathbf T}_1$
and ${\mathbf T}_4$.

\end{lem}
\begin{proof}
a) The formula (12) in \cite{BoeFJII} shows that in ${\mathcal D}_{\alpha}$
an entry of $\partial_1$ always appears together with an entry of $\partial_4$.
The same is then true for ${\mathcal D}^{\mu}$.
Furthermore, the explicit formula (\ref{satoh}) for $L^{\nu}_{\alpha+\mu}$
shows that only the contribution of $j=0$ is free
of partial derivatives w.r.t. $z_1$; it is at the
same time the only contribution free
of derivatives w.r.t. $z_4$.
\\
b) We define an element $M= M(X_1,\dots , X_n;Y_1,\dots ,Y_n)$ of
$V_{\nu}\otimes V_{\nu}$ by
$$M:= {\mathbb D}_{\alpha}(\mu,\nu) (exp tr(z_2)) ={\mathcal M} (exptr(z_2)).$$
The transformation properties of ${\mathbb D}_{\alpha}(\mu,\nu)$,
applied for
$$\left(\begin{array}{cc} A^t & 0\\
0 & A^{-1}\end{array}\right)^{\uparrow}, \quad
\left(\begin{array}{cc} A & 0\\
0 & A^{-t}\end{array}\right)^{\downarrow}\qquad (A\in GL(n,{\mathbb R}))$$
yield
$$M((X_1,\dots , X_n)\cdot A;Y_1,\dots, Y_n)= M(X_1,\dots, X_n;
(Y_1,\dots , Y_n)A^t)\quad (A\in GL(n,{\mathbb C})).$$
Such a vector in $V_{\nu}\otimes V_{\nu}$ is unique up to constants
and is therefore a scalar multiple of
$(\sum X_iY_i)^{\nu}$, i.e. $M=c\cdot (2\sum_i X_iY_i)^{\nu}$
for an appropriate constant $c=C_{\alpha}(\mu,\nu)$.\\
To understand $\mathcal M$ we study its action on those
functions on ${\mathbb H}_{2n}$,
which depend only on $z_2$; it is enough to look at functions of type
$f_T(z_2):= exp tr(Tz_2)$ with $T\in {\mathbb R}^{(n,n)},\, \det(T)\not=0$.
Then
\begin{eqnarray*}
{\mathbb D}_{\alpha}(\mu,\nu) f_T&= &
\det(T)^{-\alpha} {\mathbb D}_{\alpha}(\mu,\nu)
(f_{1_n}\mid_{\alpha}\left(
\begin{array}{cc} T & 0\\ 0 & T^{-t}\end{array}\right)\\
&=&\det(T)^{-\alpha} \left({\mathbb D}_{\alpha}
(\mu,\nu)f_{1_n}\right)\mid^{z_1}_{\alpha+\mu,\nu}
\left(\begin{array}{cc} T & 0\\ 0 & T^{-t}\end{array}\right)\\
&=& \det(T)^{\mu} c\cdot (2\sum_i X_iT^tY_i)^{\nu}\\
&=& c (D-D^{\uparrow}-D^{\downarrow})^{\nu}det(\partial_2)^{\mu} f_T.
\end{eqnarray*}
\vspace{0.3cm}
It remains to determine
the
coefficient $C_{\alpha}(\mu,\nu)$;
we compute ${\mathbb
D}_s(\mu,\nu)\det(z_2)^s$ in two ways, using the standard formulas
(see e.g. \cite[Section 1]{ BCG})
$$\det(\partial_2)\det(z_2)^s= C_n\left(\frac{s}{2}\right)\det(z_2)^{s-1}$$

$${\mathcal D}_{\alpha} \det(z_2)^s=(-1)^n C_n\left(\frac{s}{2}\right)
C_n(\alpha-n+\frac{s}{2}) \det(z_2)^{s-1}.$$
\vspace{0.3cm}
Then
$${\mathbb D}_{\alpha}(\mu,\nu) \det(z_2)^{s}= C_{\alpha}(\mu,\nu)
\left(\prod_{j=0}^{\mu-1} C_n\left(\frac{s-j}{2}\right)\right)
\{\left(D-D^{\uparrow}-D^{\downarrow}\right)^{\nu}
\det(z_2)^{s-\mu}\}\mid_{z_2=0}$$ and on the other hand
$${\mathbb D}_{\alpha}(\mu,\nu)
\det(z_2)^{s}=L^{\nu}_{\alpha+\mu}( {\mathcal D}^{\mu}_{\alpha}
\det(z_2)^s)$$
$$=
\prod_{j=0}^{\mu-1}
C_n\left(\frac{s-j}{2}\right)C_n\left(\alpha-n+\frac{s+j}{2}\right)
\{ L^{\nu}_{\alpha+\mu} \det(z_2)^{s-\mu}\}\mid_{z_2=0}$$
$$=\prod_{j=0}^{\mu-1} C_n\left(\frac{s-j}{2}\right)C_n\left(\alpha-n+\frac{s+j}{2}\right)
\frac{1}{(\alpha+\mu)^{[\nu]}\nu!}
\{\left(D-D^{\uparrow}-D^{\downarrow}\right)^{\nu}
\det(z_2)^{s-\mu}\}\mid_{z_2=0}$$
\vspace{0.4cm}
If $\nu=n\nu'$ is a multiple of $n$, then $s:=\mu+\nu'$ gives nonzero
contributions and we get
$$C_{\alpha}(\mu,\nu)=\frac{1}{(\alpha+\mu)^{[\nu]}\nu!}
\prod_{j=0}^{\mu-1} C_n \left(\alpha-n+ \frac{\mu+\nu'+j}{2}\right) .
$$
Actually, this formula makes sense (and is also valid) for arbitrary $\nu$.
\end{proof}
For the special case $n=2$ considerations similar to the above appear in
\cite[Lemma 7.5., Corollary 7.6]{DIK}.

\subsection{Doubling method with the differential operators
${\mathbb D}_{\alpha}(\mu,\nu)$} The inner product $(\sum
a_iX_i,\sum b_iX_i)=\sum a_i\overline{b_i}$
on  $V_1:={\mathbb C}[X_1,\dots
X_n]_{1}$ induces a ``produit scalaire adapt\'{e}" (see \cite{Go})
on  the $\nu$-fold
symmetric tensor product $V_{\nu}=Sym^{\nu}(V_1)={\mathbb
C}[X_1,\dots X_n]_{\nu}$ by
$$\{\alpha_1\cdot \dots \cdot \alpha_{\nu},\beta_1\cdot \dots
\cdot \beta_{\nu}\}=\frac{1}{\nu!} \sum_{\tau} \prod_{j=1}^{\nu}
(\alpha_{\tau(j)},\beta_j)\qquad (\alpha_i,\beta_j\in V_1),$$
where $\tau$ runs over the symmetric group of order $\nu$.
This inner product is invariant under the action of unitary matrices via
$Sym^{\nu}$.
\\
Note that for all ${\bf v}\in {\mathbb C}[X_1,\dots , X_n]_{\nu}$ we have
$$\left\{{\bf v},\left(\sum X_iY_i\right)^{\nu}\right\}=\tilde{\bf v},$$
where $\tilde{\bf v}$ denotes the same polynomial as ${\bf v}$,
but with the variables
$Y_i$ instead of the $X_i$.\\
We describe here the general pullback formula for level $N$
Eisenstein series ($N$ square free).
\\
We put
$$G_k^{(2n)}(Z,s)=
\sum_{M\in \Gamma^{(2n)}_0(N)_{\infty}\backslash\Gamma^{(2n)}_0(N)}
\det(CZ+D)^{-k-s}\det(C\overline{Z}+D)^{-s}.$$
For a cusp form $F\in
S_{\rho}(\Gamma_0^{(n)}(N))$ with $\rho=\det^{k+\mu}\otimes
\sigma_{\nu}$ and $z=x+iy, w=u+i{\bf v}\in \HH_n$ we get

$$ \int_{\Gamma^{(n)}_0(N)\backslash {\mathbb H}_n}
\left\{\rho(\sqrt{y})F(z),\rho(\sqrt{y}) {\mathbb
D}_{s+k}(\mu,\nu)G_k^{(2n)}\left(\left(\begin{array}{cc} z & 0\\ 0 &
-\bar{w}
\end{array}\right),\bar{s}\right)\det(y)^s\det({\bf v})^s\right\} d\omega_n$$
\begin{equation}= \gamma_n(k,\mu,\nu,s)
\sum_M F(w)\mid T_N(M) \det(M)^{-k-2s}\label{boe1} .\end{equation}
\vspace{0.3cm}
Here $d\omega_n= \det(y)^{-n-1}dxdy$, $M$
runs over all (integral) elementary divisor matrices of
size $n$ with $M\equiv 0 \bmod N$, and $T_N(M)$ denotes the Hecke
operator associated to the double coset $\Gamma^{(n)}_0(N)
\left(\begin{array}{cc} 0 & -M^{-1}
\\ M & 0\end{array}\right)\Gamma^{(n)}_0(N)$.
\\[0.3cm]
To compute the archimedean factor $\gamma$ one should keep in mind
that the unfolding of the integral leads to an integration over
$\HH_n$ involving $ {\mathbb D}_{k+s}(\mu,\nu) h_{k+s,s}$. Then
$\gamma$ is naturally a product of (essentially) three factors

$$\gamma_n(k,\mu,\nu,s)= i^{nk+n\mu+\nu} 2^{n(n-k-\mu-2s-\nu+1}
A_{k+s,\mu} B_{k+\mu+s,\nu} I(s+k+\mu-n-1,\nu) $$ with a Hua type
integral
$$I(\alpha,\nu)=
\frac{\pi^{\frac{n(n+1)}{2}}}{\alpha+n+\nu}
\prod_{j=1}^{n-1}\frac{(2\alpha+2j+1)(n+j+2\alpha)^{[\nu]}}{
(\alpha+j)\Gamma(\nu+n+j+2\alpha+1)}. $$ We refer to
\cite[Sect.3]{BSY}, see also \cite[2.2]{Boe} for details.

\subsection{Doubling method with the differential operators
${\mathbb D}_{k}(\mu,\nu)$} There are two ways to obtain holomorphic
Siegel Eisenstein series of degree $n$ and low weight after analytic
continuation (sometimes called ``Hecke summation"): One is by
evaluating at $s=0$, the other by considering $s_1=\frac{n+1}{2}-k$;
both are connected by a complicated functional equation involving
all Siegel Eisenstein series.
We need the case of weight $2$ and degree $4$, where only the Hecke summation for $s_1$ is available.\\
We first consider the general case:
In (\ref{boe1}) the differential operator ${\mathbb
D}_{k+s}(\mu,\nu)$ was applied  directly to the Eisenstein series of
``weight'' $k+s$. If we use the Hecke summation not in $s=0$ but in
$s_1:= \frac{2n+1}{2}-k$ for an Eisenstein series of degree $2n$, we
should better use a differential operator acting on the weight $k$
Eisenstein series $E_k^{(2n)}:=G_k^{(2n)}\cdot (\det \Im Z)^s$ to
get holomorphic modular forms (in particular theta series) after
evaluating in $s=s_1$. One might try to use the calculations of
Takayanagi \cite{Taka}. Note however that the results of \cite{Taka}
are applicable only for the case $\mu=0$; to incorporate the
differential operator ${\mathcal D}_k^{\mu}$ there is quite
complicated, see also \cite{Koz}. We avoid this difficulty by
observing that the  two types of differential operators are
actually not that different:\\
By
$$F\longmapsto {\mathcal D}_{k,s}(\mu,\nu)(F):=
\det(y)^s \det({\bf v})^s {\mathbb D}_{k+s}(\mu,\nu)(\det(Y)^{-s}\times
F)$$ we can define a new (nonholomorphic) differential operator
mapping  functions $F$ on $\HH_{2n}$ to ${\mathbb C}[X_1,\dots ,
X_n]_{\nu} \otimes {\mathbb C}[Y_1,\dots , Y_n]_{\nu}$ valued
functions on $\HH_{n}\times \HH_n$; this operator has exactly the
same transformation properties as ${\mathbb D}_k(\mu,\nu)$.
\\
Starting from the observation that ${\mathcal D}_{k,s}(\mu,\nu)$
maps
 holomorphic functions on $\HH_{2n}$ to nearly holomorphic functions on
$\HH_n\times \HH_n$ we get from the theory of Shimura
\cite{Shi2,Shi3} in the same way as in \cite[section 1]{BCG} an
operator identity
\begin{equation} {\mathcal D}_{k,s}(\mu,\nu)  =
\sum_{\rho_i,\rho_j} \delta^{(z_1)}_{\rho_j}\otimes
\delta_{\rho_j}^{(z_4)} \circ {\mathcal D}_s(\rho_i,\rho_j).
\label{decomp1}
\end{equation}
Here the $\rho_i,\rho_j$ run over finitely many polynomial
representations of $GL(n,{\mathbb C})$ and ${\mathcal
D}_s(\rho_i,\rho_j)$ denotes a $V_{\rho_i}\otimes V_{\rho_j}$-valued
holomorphic differential operator (a polynomial in the
$\partial_{i,j}$, evaluated at $z_2=0$; it changes the
automorphy factor $\det^k$
on $GL(2n,{\mathbb C})$ to $(\det^k\otimes \rho_1)\boxtimes
(det^k\otimes \rho_2)$ on $GL(n,{\mathbb C})\times GL(n,{\mathbb C})$).
As is usual in the theory
of nearly holomorphic functions, we have to avoid finitely many
weights $k$ here. Furthermore the $\delta_{\rho_i}, \,
\delta_{\rho_j}$ are non-holomorphic differential operators on
$\HH_n$, changing automorphy factors from $\det^k\otimes \rho$ to
$\det^{k+\mu}\otimes Sym^{\nu}$.
In the simplest case (i.e. $\rho=\det^{\mu}$, $\nu=2$ ), the operator
$\delta_{\rho}$ has the explicit form
$$\delta_{\rho}= (X_1,\dots ,X_n)\cdot\left( (\partial_{ij}) - 2i(k+\mu)
\Im(Z)^{-1}\right)
\cdot \left(\begin{array}{c} X_1 \\ \vdots \\ X_n\end{array}\right).$$
Furthermore we mention that, by invariant theory,
holomorphic differential operators ${\mathcal D}_s(\rho_i,\rho_j)$
with the transformation properties described above only exist in
the case $\rho_i=\rho_j$  see \cite{I}.
\\
If $\delta^{(z_1)}_{\rho}\otimes \delta^{(z_4)}_{\rho}$ is the
identity, then $\rho=det^{k+\mu}\otimes Sym^{\nu}$ and (at least for
$k\geq n$) ${\mathcal D}_s(\rho,\rho)$ is a scalar multiple of
${\mathbb D}_{k}(\mu,\nu)$, because the space of such differential
operators is one-dimensional. The decomposition (\ref{decomp1}) can
then be rewritten as
\begin{equation} p_s(k) {\mathcal D}_{k,s}(\mu,\nu)=d_s(k)
{\mathbb D}_k(\mu,\nu) + {\mathcal K} \label{decomp2}
\end{equation}
where $p_s(k)$ and $d_s(k)$ are polynomials in $k$ and $\mathcal K$
is a nonholomorphic differential operator with the same
transformation properties as ${\mathbb D}_k(\mu,\nu)$ and with the
additional property that ${\mathcal K}(F)$ is orthogonal to all
holomorphic cusp forms in the variables $z_1$ or $z_4$ (for any
$C^{\infty} $ automorphic form on ${\mathbb H}_{2n}$ with suitable
growth properties). Note that (\ref{decomp2}) holds now for {\it all}
weights $k$, if we request the finitely many exceptions from
(\ref{decomp1}) to be among the zeroes of $p_s(k)$. We also observe
that ${\mathcal D}_{k,s}(\mu,\nu)$ is a homogeneous polynomial of
degree $n\mu+\nu$ in the variables $(\partial_{ij})_{\mid z_2=0}$
and the entries of $y_1^{-1}$ and $y_4^{-1}$ and ${\mathcal K}$
consists only of monomials whose joint degree in $\partial_1$ and
$y_1^{-1}$ as well as in $\partial_4$ and $y_4^{-1}$ are both
positive, in particular, ${\mathcal K}$
cannot contribute monomials that only involve entries of $\partial_2$.\\
Therefore (as in \cite[(1.31)]{BCG}) we may  compare the coefficients of
$\det(\partial_2)^{\mu} \left(\sum_{i,j}
\frac{\partial}{\partial_{i,n+j}}X_iY_j   \right)^{\nu}$ on both
sides: We get

$$p_s(k) C_{k+s}(\mu,\nu)=d_s(k) C_k(\mu,\nu).$$
\vspace{0.3cm}
From this we obtain a version of the pullback formula (\ref{boe1})

$$ \int_{\Gamma^{(n)}_0(N)\backslash {\mathbb H}_n}
\left\{\rho(\sqrt{y})F(z),\rho(\sqrt{y}) {\mathbb
D}_{k}(\mu,\nu)E_k^{(2n)}\left(\left(\begin{array}{cc} z & 0\\ 0 &
-\bar{w}
\end{array}\right),\bar{s}\right)\right\} d\omega_n$$
\begin{equation}=
\frac{p_s(k)}{d_s(k)}\cdot\gamma_n(k,\mu,\nu,s) \sum_M F\mid T_N(M)
\det(M)^{-k-2s}. \label{boe2} \end{equation}
\vspace{0.3cm}
We need the result above for the pullback formula applied for a
 degree 4, weight 2 Eisenstein series at $s_1=\frac{1}{2}$:
we consider the holomorphic modular form
$${\mathcal E}_2^{(4)}:= Res_{s=s_1} E_{2}^{(4)}(Z,s) $$
Then we get for a cusp form $F\in S_{\rho}(\Gamma^{(2)}_0(N))$, with
$\rho=\det^{2+\mu}\otimes Sym^{\nu}$,
\begin{eqnarray}
\langle F, {\mathbb D}_2(\mu,\nu){\mathcal E}_2^{(4)}
(*,-\bar{w})\rangle &=&Res_{s=s_1}
 \langle F,{\mathbb D}_2(\mu,\nu) E_2^{(4)}(*,-\bar{w})\rangle \nonumber\\
&=& Res_{s=s_1} \frac{p_s(2)}{d_s(2)}
\langle F, {\mathbb D}_{2+s}(\mu,\nu) G_2^{(4)} \det(y)^s
\det(v)^s\rangle \nonumber \\
&=&
c\cdot Res_{s=s_1}\left(\sum_M  F(w)\mid T_N(M)
\det(M)^{-2-2s}\right). \label{boe3}
\end{eqnarray}
The relevant constant is then
\begin{equation}c=
\frac{C_2(\mu,\nu)}{C_{2+\frac{1}{2}}(\mu,\nu)}
\gamma_2\left(2,\mu,\nu,\frac{1}{2}\right).
\label{boe4}\end{equation}

\subsection{ Standard-L-functions at $s=1$ and $s=2$,
in particular for Yoshida lifts of degree 2}

\subsubsection{An Euler product}

If $F\in S_{\rho}(\Gamma_0^{(n)}(N))$ is an eigenform of all the
Hecke operators $T_N(M)$ with eigenvalues $\lambda_N(M)$, then the
Dirichlet series of these eigenvalues can be written in terms of the
(good part of) the standard $L$-function $D^{(N)}_F(s)$:

\begin{eqnarray*}
\lefteqn{\sum \lambda_N(M)\det(M)^{-s}=}\\
&& \left(\sum_{\det(M)\mid N^{\infty}} \det(M)^{-s}\right)\times
\frac{1}{\zeta^{(N)}(s)\prod_{i=1}^n
\zeta^{(N)}(2s-2i)}D_F^{(N)}(s-n).
\end{eqnarray*}

\vspace{0.4cm}
The integral representations
studied above allow
us to {investigate
(for degree 2) the}
behavior of such a standard $L$-function at
$s=1$ and $s=2$; we remark that $s=1$ is not a critical value for
the standard L-function! Note that in the formula above, we get
$D_F(1)$ for degree $n=2$ for $s=3$. In the formula (\ref{boe3}) this corresponds
to $s=s_1=\frac{1}{2}$ due to several shifts ($2s_1+2-2=1$ for this $s_1$).
\\
If $F$ is actually a Yoshida lift of level $N$ associated to two
elliptic cuspidal newforms $f\in S_{k'}(\Gamma_0(N)), g\in
S_k(\Gamma_0(N))$, with $k'\geq k$, then $F\in
S_{\rho}(\Gamma^{(2)}_0(N))$ with
$\rho=\det^{2+\frac{k'-k}{2}}\otimes \Sym^{k-2}$ is indeed an
eigenform of all the Hecke operators $T_N(M)$:
\begin{equation}\sum_M F\mid T_N(M)\det(M)^{-s}=\frac{\lambda}{N^{ns}}
\zeta^{(N)}(s-2)L^{(N)}\left(f\otimes g,s+\frac{k'+k}{2}-3\right)
\Lambda_N(s-2)    \cdot  F \label{boe5}
\end{equation}
where $\lambda=\pm N^{n(n-1)/2}=\pm N$ (with the sign depending only
on $N$),
$$\Lambda_N(s)=\prod_{p\mid N}\prod_{j=1}^2 (1-p^{-s-2+j})^{-1}$$
and
$$L^{(N)}(f_1\otimes f_2,s):= \prod_{p\nmid N} (1-\alpha_p\beta_p p^{-s})
(1-\alpha_p\beta_p'p^{-s})(1-\alpha_p'\beta_p
p^{-s})(1-\alpha_p'\beta_p'p^{-s}).$$
Moreover $F\mid_{\rho} \left(\begin{array}{cc} 0_2 & -1_2\\
N\cdot 1_2 & 0_2 \end{array}\right)$ is also an eigenfunction of all
the $T_N(M)$ with the same eigenvalues as $F$; for details on the
facts mentioned above we refer to \cite{BS1,BS3}.

\subsubsection{ A version of the pullback formula
for the Eisenstein series attached to the cusp zero}
We can consider the same doubling method using the Eisenstein series
\begin{eqnarray*}\mathfrak{F}_k^{(2n)}(Z,s)&:=& \sum_{C,D} \det(CZ+D)^{-k-s}\det(C\bar{Z}+D)^{-s},\\
{\mathbb F}_k^{(2n)}(Z,s)&:=&\mathfrak{F}_k^{(2n)}(Z,s)\times
\det(Y)^s,\end{eqnarray*} where $(C,D) $ runs over non-associated
coprime symmetric pairs with the additional condition ``$\det(C)$
coprime to N'' (this is the Eisenstein series ``attached to the cusp
zero''). The reason for using both versions is that in our previous
papers \cite{BS1,BS3} we mainly worked with $E_k^{(2n)}$, whereas the Fourier
expansion is more easily accessible for the Eisenstein series
${\mathbb F}_k^{(2n)}$.
\\
The two doubling integrals are linked to each other by the elementary relation
$$E_k^{(2n)}(Z,s)\mid_k \left(\begin{array}{cc} 0_{2n} & -1_{2n}\\
N\cdot 1_{2n} & 0_{2n}\end{array}\right)= N^{-kn-2ns} {\mathbb
F}_k^{(2n)}(Z,s).$$ Due to this relation, substituting
$\mathfrak{F}$ for $E$ in the doubling method just means (for
Yoshida-lifts) a modification by a power of $N$ (the factor $N^{-ns}$
in (\ref{boe5}) goes away).
For the case of arbitrary cusp forms we refer to \cite{BCG,BKS}.\\
We write down the relevant cases
explicitly for the Yoshida lift $F$ from above:\\
The residue of the standard $L$-function at $s=1$ corresponds to
a near center value for $L(f_1\otimes f_2,s)$:\\
The equation (\ref{boe3}) then becomes (with ${\mathcal F}_2^{(4)}:=
Res_{s=\frac{1}{2}} {\mathbb F}_2^{(4)}$)

\begin{equation}\label{ready_to_use_pullback_for_F_weight_2}
{\left\langle F,{\mathbb D}_2\left(\frac{k'-k}{2}, k-2\right)
{\mathcal F}_2^{(4)}(*,-\bar{w})\right\rangle}\end{equation} $$= c
\lambda \prod_{p\mid N} (1-p^{-1}) \Lambda_N(1)
\frac{1}{\zeta^{(N)}(3)\zeta^{(N)}(4)\zeta^{(N)}(2)}
L^{(N)}\left(f\otimes g,\frac{k'+k}{2}\right)\cdot F(w)$$
with
$$c= \frac{ C_2(\frac{k'-k}{2},k-2)}
{C_{2+\frac{1}{2}}(\frac{k'-k}{2},k-2)} \cdot \gamma_2\left(
\frac{k'-k}{2},k-2,\frac{1}{2}\right) .$$

To treat the critical value of the standard $L$-function at $s=2$,
we can directly use the formula (\ref{boe1}), taking tacitly into
account that ${\mathcal F}_4^{(4)}(Z):=
\mathbb{F}_4^{(4)}(Z,s)\mid_{s=0}$ defines a holomorphic modular
form (see \cite[Prop.10.1]{Shi1}) by Hecke summation.
\\
This yields
\begin{eqnarray}\label{pullback4}
\lefteqn{\left\langle F,{\mathbb D_4}\left(\frac{k'-k}{2}-2,
k-2\right) {\mathcal F}^{(4)}_4(*,-\overline{w})\right\rangle }
\\
&=&\gamma_2\left(4,\frac{k'-k}{2}-2,k-2,0\right)\times \nonumber\\
&& (\pm N) \,\,\Lambda_N(2)
\frac{\zeta^{(N)}(2)}{\zeta^{(N)}(4)\zeta^{(N)}(6)\zeta^{(N)}(4)}
L^{(N)}\left(f\otimes g, \frac{k'+k}{2}+1\right)\cdot F(w).\nonumber
\end{eqnarray}
In the case of a general cusp form $F\in
S_{\rho}(\Gamma^{(2)}_0(N))$, which we assume to be an eigenfunction
of the Hecke operators ``away from $N$'', we can write
\begin{eqnarray*}
\lefteqn{\left\langle F,{\mathbb
D_4}\left(\frac{k'-k}{2}-2\right){\mathcal
F}^{(4)}_4(*,-\overline{w})\right\rangle }
\\
&=&\gamma_2\left(4,\frac{k'-k}{2}-2,k-2,0\right)\times
\frac{D_f^{(N)}(2)}{\zeta^{(N)}(4)\zeta^{(N)}(6)\zeta^{(N)}(4)}
{\mathcal T}(F)(w)\end{eqnarray*} where ${\mathcal T}$ is an
(infinite) sum  of Hecke operators at the bad places.

\section{ Integrality properties}
The known results about integrality of Fourier coefficients of
Eisenstein series are not sufficient for our purposes because they
deal only with level one and large weights. We do not aim at the
most general case, but just describe how to adapt the reasoning in
\cite[section 5]{B4} to the cases necessary for our purposes.

\subsection{ The Eisenstein series}

We collect some facts about the Fourier coefficients of Eisenstein
series

$$F_k^m(Z):= {\mathbb F}_k^m(Z,s)_{\mid s=0}$$
for even  $m=2n$ with $k\geq \frac{m+4}{2}$
\\
This function is known to define a holomorphic modular form with
Fourier expansion

$$F_k^m(Z)=\sum_{T\geq 0} a^m_k(T,N)exp(2\pi i tr(TZ)).$$
We first treat $T$ of maximal rank. We denote by $d(T):=
(-1)^n\det(2T)$ the discriminant of $T$ and by $\chi_T$ the
corresponding quadratic character, defined by
$\chi_T(.):=\left({d(T)\over .}\right)$.
\\
Then $a^m_k(T,N)=0$  unless $T>0$, see e.g. \cite[prop.5.2]{BCG}.
\\
If $T>0$ then the Fourier coefficient is of type
$$a^m_k(T)=A^m_k \det(T)^{k-\frac{m+1}{2}}\prod_{p\nmid N} \alpha_p(T,k)$$
where $\alpha_p(T,k)$ denotes the usual local singular series
and

$$A^m_k=(-1)^{\frac{mk}{2}} \frac{2^m}{\Gamma_m(k)}\pi^{mk}.$$

We can express the nonarchimedean part by a normalizing factor and
polynomials in $p^{-k}$:

\begin{eqnarray*}\prod_{p\nmid N}\alpha_p(T,k)&=&
\frac{1}{\zeta^{(N)}(k)\prod^n_{j=1}\zeta^{(N)}(2k-2j)}\times\\
&&\sum_{G} det(G)^{-2k+m-1} L^{(N)}(k-n,\chi_{T[G^{-1}]})
\prod_{p\nmid N}\beta_p(T[G^{-1}],k).\end{eqnarray*}
Here $G$ runs over
$$GL(n,{\mathbb Z})\backslash
\{M\in {\mathbb Z}^{(n,n)}\,|\, \det(M)\,\,\mbox{coprime
to}\,\,N\}$$ and the $\beta_p(T)$ denote the ``normalized primitive
local densities''. In general  they are polynomials in $p^{-k}$ with
integer coefficients and they are equal to one
for all $p$ coprime to $d(T)$, see e.g. \cite[section 2]{B4}.\\
Let $f_T$ be the conductor of the quadratic character $\chi_T$ and
$\eta_T$ the corresponding primitive character. Then

\begin{eqnarray*}
L^{(N)}(k-n,\chi_T)&=&\prod_{p\mid N}(1-\chi_T(p)p^{-k+n})L(k-n,\chi_T)\\
&=& \prod_{p\mid N} (1-\chi_T(p)p^{-k+n})\prod_{p\mid d(T)}
(1-\eta(p)p^{n-k}) L(k-n,\eta_T)
\end{eqnarray*}
We quote from \cite{B4} that
$$\left(\frac{d(T)}{f_T}\right)^{k-\frac{m}{2}}
\prod_{p\nmid N} (1-\eta_T(p)p^{n-k})\beta_p(T,k)\in {\mathbb Z}.$$
We may therefore just ignore this factor. Then as in \cite{B4} we
use the functional equation of the Riemann zeta function and the
Dirichlet L-functions attached to quadratic characters.
\\
We get (for $4\mid k$) that
$$a_k^m(T,N)\in \prod_{p\mid N}
\left((1-p^{-k})\prod_{j=1}^n (1-p^{-2k+2j})\right) 2^n
\frac{k}{B_k} \frac{1}{{\mathcal N}^*_{2k-m}}\prod_{j=1}^n
\frac{k-j}{B_{2k-2j}} \cdot \frac{1}{N^{k-n}}\cdot {\mathbb Z}$$
Here the factor $N^{k-n}$ takes care of the possible denominator
arising from \\
$\prod_{p\mid N}(1-\chi_T(p)p^{-k+n})$ and

$${\mathcal N}_{2k-m}^*:= \prod_{p\mid {\mathcal N}_{2k-m}} p^{1+\nu_p(k-n)},$$
where ${\mathcal N}_{2k-m}$ is the denominator of the Bernoulli
number $B_{2k-m}$.
\\
If $k\equiv 2\bmod 4$ there is a similar formula, see \cite{B4}.
\\[0.3cm]
We have to assure that nonzero Fourier coefficients of lower rank do
not occur. This is a classical fact in the range of absolute convergence
(i.e. $k>m+1$), see e.g. the calculations in \cite[section 18]{Ma}.
It is also true for small weight $k\geq \frac{m+4}{2}$ and level one,
as shown by Haruki \cite[Theorem 4.14]{Har}; his result relies
on calculations by Shimura \cite{Shi1} and Mizumoto \cite{Miz}.
The basic ingredient for Haruki is an expression \cite[(1.1)]{Har}
for Fourier coefficients $T$ of rank $r<m$ as finite sums of products
of $\Gamma$-factors, singular series,
confluent hypergeometric functions and Eisenstein series  for
$Gl(n)$
evaluated at $s=0$. Haruki's procedure remains valid
for level $N>1$ as long as it is based on individual vanishing of the products
mentioned above (the modification for level $N>1$ means to omit the local
singular series for primes dividing $N$
, i.e. for all $p\mid N$ one has to multiply the level one expression by
a polynomial in $p^{-s-2k}$, evaluated at $s=0$).
Indeed, as shown in the proof of Theorem 4.14 \cite{Har},
such individual vanishing occurs for all $T$ of rank $r<m$ and all
$k\geq \frac{m+4}{2}$ except possibly for the case $k=\frac{m+4}{2}$
and $r=m-4>0$; in this exceptional case the vanishing for level one depends
on cancellations for some $T$.\\
{\it In summary, the Fourier coefficients $a^m_k(T,N)$
all vanish for $rank(T)<m$
and $k>\frac{m+4}{2}$ and also for $m=k=4$.}

\begin{remar} The Fourier coefficients of $F_4^4(N)$ are in
$$   \prod_{p|N}\left((1-p^{-4})^2(1-p^{-6})\right)
\frac{9}{2N^2}\cdot  {\mathbb Z}\subseteq \frac{9}{N^{16}} \cdot
{\mathbb Z}\left[\frac{1}{2}\right].$$
\end{remar}

\subsection{ The differential operators}

By definition, the coefficients of the differential operator
${\mathcal D}_k^{\mu}$ are in ${\mathbb Z}[1/2]$; here we view
${\mathcal D}_k^{\mu}$ as a
polynomial in the variables $z_2$ and $\partial_{ij}$.\\
Concerning the integrality properties of $L^{\nu}_k$, we just remark
that because of
$$(2-k-\nu)^{[j]}=(-1)^{j}(k+\nu-j-1)^{[j]}=
(-1)^j\frac{(k+\nu-2)!}{(k+\nu-j-1)!}$$ it is sufficient to look at
$$\frac{(k+\nu-j-1)!}{k^{[\nu]} j!(\nu-2j)! (k+\nu-2)!}\qquad
(0\leq j\leq \left[\frac{\nu}{2}\right]).$$
Taking into account that
$\frac{\nu!}{j!(\nu-2j)!}\in {\mathbb Z}$ and
$$\frac{(k+\nu-j-1)!}{(k+\nu-2)!}
\in \frac{1}{(k+\nu-[\frac{\nu}{2}])... (k+\nu-2)}\cdot {\mathbb
Z}$$
we see that the coefficients of $L^{\nu}_k$ are in
$$\frac{1}{k^{[\nu]}\nu! (k+\nu-2)... (k+\nu-[\frac{\nu}{2}])}\cdot
{\mathbb Z}.$$
Putting things together, we see that
 ${\mathbb D}_k(\mu,\nu)$ has coefficients in
$$\frac{1}{(k+\mu)^{[\nu]}\nu!
(k+\mu+\nu-2)... (k+\mu+\nu-[\frac{\nu}{2}])} \cdot {\mathbb
Z}\left[1/2\right].$$

\begin{remar} The Fourier coefficients of
${\mathbb D}_4(\mu,\nu) {\bf F}^4_4$ are in
$$\frac{1}{(4+\mu)^{[\nu]}\nu!
(4+\mu+\nu-2)... (4+\mu+\nu-[\frac{\nu}{2}])}\times \frac{9}{N^{16}}
{\mathbb Z}\left[1/2\right].$$
\end{remar} This remark does not claim, that the denominator given there is the
best possible one, there may be additional cancellations of
denominators coming from the restriction.

\section{The Petersson norm of the Yoshida lift}
Take $f=\sum a_nq^n, g=\sum b_nq^n$ as in the introduction, of
weights $k'$ and $k$ respectively and assume that for all primes $p$
dividing the common (square-free) level $N$ of $f,g$ both functions
have the same Atkin-Lehner eigenvalue $\epsilon_p$. Let
$k'=2\nu_1+2, k=2\nu_2+2$. Choose a factorization $N=N_1N_2$, where
$N_1$ is the product of an odd number of prime factors, and let
$D=D(N_1,N_2)$ be the definite quaternion algebra over $\QQ$,
ramified at $\infty$ and the primes dividing $N_1$. Let
$R=R(N_1,N_2)$ be an Eichler order of level $N=N_1N_2$ in
$D(N_1,N_2)$ with (left) ideal class number $h$.

We recall (and slightly modify) some  notation from  \S 1 of
\cite{BS3}:
For $\nu \in \N$ let $U_{\nu}^{(0)} $ be the space of homogeneous
harmonic polynomials of degree $\nu$ on $\R^3$ and view $P \in
U_{\nu}^{(0)}$ as a polynomial on $D_\infty^{(0)}= \{ x \in D_\infty
\vert \tr (x)=0\}$ by putting
$P(\sum_{i=1}^{3}x_ie_i)=P(x_1,x_2,x_3)$ for an orthonormal basis
$\{ e_i\}$ of $D_\infty^{(0)}$ with respect to the norm form $n$ on
$D$.
The representations $\tau_{\nu}$ of $D_\infty^\times/\R^\times$ of
highest weight $(\nu)$ on $U_{\nu}^{(0)}$ given by
$(\tau_{\nu}(y))(P)(x)=P(y^{-1}xy)$ for $\nu \in \N$ give all the
isomorphism classes of irreducible rational representations of
$D_\infty^\times/\R^\times.$

For an irreducible rational representation $(V_{\tau},\tau)$ (with
$\tau=\tau_\nu$ as above) of $D^\times_\infty/\R^\times$ we denote
by $\cA(D^\times_{\bA},R^\times_{\bA},\tau)$ the space of functions
$\phi: D^\times_{\bA} \to V_\tau$ satisfying $\phi(\gamma x
u)=\tau(u^{-1}_\infty)\phi(x)$ for $\gamma\in D^\times_{\Q}$ and
$u=u_\infty u_f\in R^\times_{\bA}$, where $R^\times_{\bA} =
D^\times_\infty\times \prod_p R^\times_p$ is the adelic group of
units of $R$.
Let $D^\times_{\bA}=\cup^r_{i=1}D^\times y_i R^\times_{\bA}$ be a
double coset decomposition with $y_{i,\infty} = 1$ and $n(y_i)=1$. A
function in $\cA(D^\times_{\bA},R^\times_{\bA},\tau)$ is then
determined by its values at the $y_i$. We  put
$I_{ij}=y_iRy_j^{-1},\ R_i=I_{ii}$ and let $e_i$ be the number of
units of the order $R_i$.
On the space $\cA(D^\times_{\bA},R^\times_{\bA},\tau)$ we have for
$p\nmid N$ Hecke operators $\tilde T(p)$ defined by $\tilde
T(p)\phi(x) =
\mathop\int\limits_{D^\times_p}\phi(xy^{-1})\chi_p(y)dy$ where
$\chi_p$ is the characteristic function of $\lbrace y\in R_p\vert
n(y)\in p\Z^\times_p\rbrace$. They commute with the involutions
$\tilde w_p$ and are given explicitly by $\tilde T(p)\phi(y_i) =
\mathop\sum\limits^r_{j=1} B^\nu_{ij}(p)\phi(y_j)$, where the Brandt
matrix entry $B^\nu_{ij}(p)$ is given as
$$B_{ij}(p)\ = B_{ij}^{(\nu)}(p)= {1\over{e_j}} \quad
{\mathop\sum\limits_ {{x\in y_j Ry^{-1}_i}\atop{n(x)=p}}} \quad
\tau(x) \ ,$$ hence is itself an endomorphism of the representation
space $U_\nu^{(0)}$ of $\tau.$

From \cite{E,H-S,Shz,J-L} we know then that the essential part
$\cA_{\text{ess}} (D^\times_{\bA},R^\times_{\bA},\tau)$ consisting
of functions $\phi$ that are orthogonal (under the natural inner
product) to all
$\psi\in\Ascr(D^\times_{\bA},(R'_{\bA})^\times,\tau)$ for orders
$R'$ strictly containing $R$ is invariant under the $\tilde T(p)$
for $p\nmid N$ and the $\tilde {w}_p$ for $p{\nmid} N$ and hence has
a basis of common eigenfunctions of all the $\tilde T(p)$ for
$p\nmid N$. Moreover in $\cA_{\text{ess}}(D^\times_{\bA}$,
$R^\times_{\bA},\tau)$ strong multiplicity one holds, i.e., each
system of  eigenvalues of the $\tilde T(p)$ for $p\nmid N$ occurs at
most once, and the eigenfunctions are in one to one correspondence
with the newforms in the space $S^{2+2\nu}(N)$ of elliptic cusp
forms of weight $2+2\nu$ for the group $\Gamma_0(N)$ that are
eigenfunctions of all Hecke operators (if $\tau$ is the trivial
representation and $R$ is a maximal order one has to restrict here
to functions orthogonal to the constant function $1$ on the
quaternion side in order to obtain cusp forms on the modular forms
side).

Let $\phi_1=\phi_1^{(N_1,N_2)}:D^\times_{\AAA}\rightarrow
U_{\nu_1}^{(0)}$ and
$\phi_2=\phi_2^{(N_1,N_2)}:D^\times_{\AAA}\rightarrow
U_{\nu_2}^{(0)}$ correspond to $f$ and $g$ respectively with respect
to the choice of $N_1,N_2$ and hence of $D=D(N_1,N_2)$. Let
$F=F_{f,g}=F_{\phi_1,\phi_2}$ (which of course also depends on the
choice of $N_1,N_2$) be the Yoshida lift; it takes values in the
space $W_\rho$ of the symmetric tensor representation
$\rho=\det^{\kappa}\otimes\Sym^j(\CC^2)$, $j=k-2,
\kappa=2+\frac{k'-k}{2}$ and is a Siegel cusp form $F\in
S_{\rho}(\Gamma_0^{(2)}(N))$.
To describe it explicitly we notice that the group of proper
similitudes of the quadratic form $q(x)=n(x)$ on $D$ (with
associated symmetric bilinear form $B(x,y)=\text{tr}(x \bar y)$,
where ${\rm tr}$ denotes the reduced trace on $D$) is isomorphic to
$(D^\times\times D^\times)/Z(D^\times)$ (as algebraic group) via
$(y,y')\mapsto \sigma_{y,y'}$  with
$\sigma_{y,y'}(x)  = y x (y')^{-1} ,$  the special orthogonal
group is then the image of
$\{(y,y')\in D^\times\times D^\times\mid n(y)= n(y')\}.$

We denote by $H$ the orthogonal group of $(D,n)$, by $H^+$ the
special orthogonal group and by $K$ (resp. $ K^+)$ the group of
isometries (resp.\
isometries of determinant $1$) of the lattice  $R$ in $D$.
It is well known that the
$H^+(\R)$-space $U_{\nu_1}^{(0)}\otimes U_{\nu_2}^{(0)}$ is
isomorphic to the $H^+(\R)$-space $U_{\nu_1,\nu_2}$ of $\C[X_1,
X_2]$-valued harmonic forms on $D_{\infty}^2$ transforming according
to the representation of $GL_2(\R)$ of highest weight
$(\nu_1+\nu_2,\nu_1-\nu_2)$; an intertwining map $\Psi$ has been
given in \cite[Section 3]{BS5}. It is also well known \cite{K-V}
that the representation $\lambda_{\nu_1,\nu_2}$ of $H^+(\R)$ on
$\Unn$ is irreducible of highest weight $(\nu_1+\nu_2,\nu_1-\nu_2)$.
If $\nu_1>\nu_2$ it can be extended in a unique way to an
irreducible representation of $H(\R)$ on the space
$U_{\nu_1,\nu_2,s}:=(U^{(0)}_{\nu_1}\otimes U^{(0)}_{\nu_2}) \oplus
(U^{(0)}_{\nu_2}\otimes U^{(0)}_{\nu_1})=:U_\lambda$ which we denote
by $(\tau_1\otimes \tau_2)=:\lambda$ for simplicity, on this space
$\sigma_{y,y'}\in H^+(\R)$ acts via $\tau_1(y)\otimes \tau_2(y')$ on
the summand $U^{(0)}_{\nu_1}\otimes U^{(0)}_{\nu_2}$ and via
$\tau_2(y)\otimes \tau_1(y')$ on the summand $U^{(0)}_{\nu_2}\otimes
U^{(0)}_{\nu_1}$.
For $\nu_1=\nu_2$ there are two possible extensions to
representations $(\tau_1\otimes \tau_2)_{\pm}$ on $\Unn$; we denote
this space with the representation $(\tau_1\otimes
\tau_2)_{+}=:\lambda$ on it by $U_\lambda$ again (and don't consider
the minus variant in the sequel).

We recall then from \cite{K-V, Weiss_vektor, BS3} that  the space
${\mathcal H}_q(\rho)$ consisting of all $q$-pluriharmonic
polynomials $P:M_{4,2}(\CC) \rightarrow W_{\rho}$ such that $ P(xg)
= (\rho(g^t))P(x)$ for all $g \in GL_2(\CC)$ is isomorphic to
$(U_\lambda, \lambda)$ as a representation space of $H(\RR)$.
The space ${\mathcal H}_q(\rho)$ carries an essentially unique
$H(\RR)$-invariant scalar product $\langle
\quad,\quad\rangle_{{\mathcal H}_q(\rho)}$, and in the usual way we
can find a reproducing $H(\RR)$ invariant kernel $P_{\rm Geg} \in
{\mathcal
  H}_q(\rho)\otimes {\mathcal H}_q(\rho)$ (generalized Gegenbauer
polynomial) , i.\ e., $P_{\rm Geg}$ is a polynomial on
$D_\infty^2\oplus D_\infty^2$ taking values in $W_\rho\otimes
W_\rho$ which as function of each of the variables
\begin{itemize}
\item[i)] is a
$q$-pluriharmonic polynomial in ${\mathcal H}_q(\rho)$,
\item[ii)] is
symmetric in both variables
\item[iii)] satisfies   $P_{\rm Geg}(h{\bf
x},h\tilde{\bf x})=P_{\rm Geg}({\bf
  x},\tilde{\bf x})$ for $h \in H(\RR)$
\item[iv)] satisfies
$\langle P_{\rm Geg}({\bf x},\cdot),P(\cdot)\rangle_{{\mathcal
H}_q(\rho)}=P({\bf
  x})$ for all $P\in
{\mathcal H}_q(\rho)$.
\end{itemize}
In fact, since such a polynomial is
characterized by the first three properties up to scalar multiples
we can construct it (in a more general situation) with the help of
the differential operator ${\mathbb
  D}_{\alpha}(\mu,\nu)$ and the
polynomial $Q_\alpha^{\mu,\nu}$ from \ref{definition_diffop_prop}:

 For $k\in {\mathbb N}$ and nonnegative integers $\mu,\nu$ we
define a polynomial map
 $$ \widetilde{P_{\rm Geg}}^{(k,\mu,\nu)} :
{\mathbb C}^{2k,n} \times {\mathbb C}^{2k,n}\longrightarrow
V_{\nu}\otimes V_{\nu}$$ by
$$\widetilde{P_{\rm Geg}}^{(k,\mu,\nu)}({\bf Y_1},{\bf Y_2}):=
Q^{(\mu,\nu)}_{k}\left(\left(\begin{array}{cc}
{\bf Y}_1^t{\bf Y} &  {\bf Y}_1^t{\bf Y}_2\\
{\bf Y}_2^t{\bf Y}_1 & {\bf Y}_2^t{\bf Y}_2\end{array}\right)
\right)$$ Then $\widetilde{P_{\rm Geg}}^{(k,\mu,\nu)}$ is symmetric
and pluriharmonic in ${\bf Y_1 }$ and ${\bf Y_2}$, see \cite{I} ;
moreover,  for $A,B\in GL(n,{\mathbb C})$ we have
$$\widetilde{P_{\rm Geg}}^{(k,\mu,\nu)}({\bf Y_1}\cdot A,{\bf Y_2}
\cdot
B)=\det(A)^{\mu}\det(B)^{\mu}\sigma_{\nu}(A)\otimes\sigma_{\nu}(B)
(\widetilde{P_{\rm Geg}}^{(k,\mu,\nu)}({\bf Y}_1,{\bf Y}_2).$$ For
$g\in O(2k,{\mathbb C})$ we get
$$\widetilde{P_{\rm Geg}}^{(k,\mu,\nu)}(g{\bf Y}_1,{\bf Y }_2)=
\widetilde{P_{\rm Geg}}^{(k,\mu,\nu)}({\bf Y}_1,g^{-1}{\bf Y}_2).$$

If we consider a $2k$-dimensional positive definite real quadratic
space with positive definite quadratic form $q$ and associated
bilinear form $B$ (so that $B(x,x)=2q(x)$) we write
$q(x_1,\ldots,x_{2n})=(B(x_i,x_j)/2)_{i,j}$ for (half) the $2n\times
2n$ Gram matrix associated to the $2n$-tuple of vectors
$(x_1,\ldots,x_{2n})$ and put in a similar way as above for
$(\by,\by')\in V^{2n}$
$$P_{\rm Geg}^{(k,\mu,\nu)}(\by,\by')=Q_k^{\mu,\nu}(q(\by,\by')),$$
this gives a nonzero polynomial with values in $V_\nu\otimes V_\nu$
which is symmetric in the variables $\by,\by'$, is $q$-pluriharmonic
in each of the variables with the proper transformation under the
right action of $GL_n$ and is invariant under the diagonal action of
the orthogonal group of $q$; it is hence a scalar multiple of the
$V_\nu\otimes V_\nu$-valued Gegenbauer polynomial on this space.

If we apply the differential operator ${\mathbb D}_k(\mu,\nu)$ to a
degree $2n$ theta series $\Theta^{2n}_S(Z):=\sum_{R\in {\mathbb
Z}^{2k,2n}} \exp 2\pi i \tr(R^tSRZ)$ written in matrix notation we
get
\begin{equation*}\label{theta_differentiated}
\begin{split}
({\mathbb D}_k&(\mu,\nu) \Theta^{2n}_S)(z_1,z_4)\\
&=\sum_{R_1,R_2\in {\mathbb Z}^{(2k,n)}} (2\pi i)^{n\mu}
Q_k^{(\mu,\nu)}\left(\begin{pmatrix}
S[R_1] & R_1^tSR_2\\
R_2^tSR_1 & S[R_2]\end{pmatrix}\right) \exp 2\pi i
\tr(S[R_1]z_1+S[R_2]z_4);
\end{split}\end{equation*}
writing the theta series in lattice notation as the degree $2n$
theta series
\begin{equation*}
\theta_{\Lambda}^{(2n)}(Z)=\sum_{\bx \in \Lambda^{2n}}\exp(2\pi i
\tr (q(\bx)Z))
\end{equation*}
of a lattice $\Lambda$ on $V$ we obtain in the same way
\begin{equation}\label{theta_gegenbauer_diffops}
\begin{split}
{\mathbb D}_k(\mu,\nu)\theta_{\Lambda}^{(2n)}&(z_1,z_4)\\ &= (2\pi
i)^{n\mu}\sum_{(\by,\by')\in
  \Lambda^{2n}}P_{\rm Geg}(\by,\by')\exp(2\pi i \tr
(q(\by)z_1+q(\by')z_4))\\
&= (2\pi i)^{n\mu}\sum_{(\by)\in
  \Lambda^{n}}\theta_{\Lambda}^{(n,\nu)}(z_4)(\by)\exp(2\pi i \tr (q(\by)z_1)),
\end{split}
\end{equation}

where we have written
\begin{equation}\label{theta_gegenbauer_defi}
\theta_{\Lambda}^{(n,\nu)}(z_4)(\by):=\sum_{(\by')\in
  \Lambda^{n}}P_{\rm Geg}(\by,\by')\exp(2\pi i \tr
(q(\by')z_4)).
\end{equation}

Going through the construction above in our quaternionic situation
with $V_\nu=W_\rho$ we see that we can normalize the scalar product
on  ${\mathcal H}_q(\rho)$ in such a way that the polynomial $P_{\rm
Geg}$ obtained in the way just described is indeed the reproducing
kernel for this space. We choose this normalization in what follows
and write
\begin{equation*}
\theta_{ij,\rho}(Z)(\tilde{{\bf x}}):= \mathop\sum\limits_{{{{\bf
x}}}\in (y_i Ry^{-1}_j)^2} P_{\rm Geg}({\bf x},\tilde{\bf x})
\exp(2\pi i\tr(q({\tilde{\bf x}})Z))\in W_\rho\otimes W_\rho
\end{equation*}
(so that $\theta_{ij,\rho}(Z)$ is (for each $Z$ in the Siegel upper
half space ${\mathfrak H}_2$) an element of ${\mathcal H}_q(\rho)
\otimes W_\rho$.) For an arbitrary lattice $\Lambda$ on $D$ the
theta series $\theta_{\Lambda,\rho}$ is defined analogously as given
in equation (\ref{theta_gegenbauer_defi}).

\par\noindent
We denote by ${\mathcal P}$ the (essentially unique) isomorphism
from $U_{\lambda}$ to ${\mathcal H}_q(\rho)$. With the help of the
map $\Psi$ from \cite{BS5} mentioned above we can fix a
normalization and write ${\mathcal P}(R_1\otimes R_2)$ for $R_j \in
U_{\nu_j}^{(0)}$ as
\begin{equation}\label{psi_definition}
\begin{split}{\mathcal P}(R_1&\otimes
R_2)(d_1,d_2)(X_1,X_2)\\&=
({\mathcal
  D}(n(d_1X_1+d_2X_2)^{\nu_2}\tau_2(d_1X_1+d_2X_2)R_2)R_1)(\Im(d_1\overline{d_2})),\end{split}
\end{equation}
where we associate as usual to a polynomial $R\in \C[t_1,t_2,t_3])$
the differential operator ${\mathcal
D}(R)=R(\frac{\partial}{\partial
  t_1},\frac{\partial}{\partial t_2},\frac{\partial}{\partial t_3})$,
set $\Im(d)=d-\bar{d}$ and write all vectors as coordinate vectors
with respect to an orthonormal basis.

\begin{defi} With notation as above we define the Yoshida lift of
  $(\phi_1,\phi_2)$, or also of $(f,g)$ with respect to $(N_1,N_2)$, to
be given by
\begin{equation*}
F(Z):=Y^{(2)}(\phi_1,\phi_2)(Z):=\sum^r_{i,j{=}1} {1\over{e_i e_j}}
\langle {\mathcal P}( \phi_1(y_i)\otimes \phi_2(y_j)),
\theta_{ij,\rho}(Z)\rangle_{{\mathcal
    H}_q(\rho)} \in W_\rho.
\end{equation*}
\end{defi}
\begin{lem}\label{symmetry_of_ylift}
\begin{enumerate}
\item One has $\theta_{ij,\rho}(Z)({\bf
  x})=\theta_{ji,\rho}(Z)(\bar{\bf x})$ (where $\bar{\bf
  x}=(\bar{x_1},\bar{x_2})$ denotes the quaternionic conjugate of the
pair ${\bf x}=(x_1,x_2)$).
\item
\begin{eqnarray*}
2F(Z)&=&2Y^{(2)}(\phi_1,\phi_2,Z)\\
&=&\sum^r_{i,j{=}1} {1\over{e_i e_j}} \mathop\sum\limits_{{{{\bf
x}}}\in (y_i Ry^{-1}_j)^2} {\mathcal P}( \phi_1(y_i)\otimes
\phi_2(y_j)+\phi_2(y_i)\otimes
\phi_1(y_j)))(x_1,x_2)\times\nonumber\\
&& \quad\quad\quad\quad\quad\quad\quad\quad\quad\quad\quad
\times\exp(2\pi i\tr(q({{\bf x}})Z)).\nonumber
\end{eqnarray*}
\item
Denote by  $\langle F,\theta_{ij,\rho}\rangle_{\textup{Pet}}$ the
Petersson product of the vector valued Siegel modular forms $F$ and
$\theta_{ij,\rho}$. Then the function $\xi: (y_i,y_j) \mapsto
\langle F,\theta_{ij,\rho}\rangle_{\textup{Pet}}\in {\mathcal
H}_q(\rho)$ has the symmetry property $\xi(y_i,y_j)({\bf
x})=\xi(y_j,y_i)(\bar{{\bf x}})$. It induces  a unique function,
denoted by $\tilde{\xi}$, on $H(\bA)$ satisfying
$\tilde{\xi}(\sigma_{y_i,y_j})=\xi(y_i,y_j)$ and
\begin{equation*}
\tilde{\xi}(\gamma \sigma k)=\lambda
(k_\infty^{-1})\tilde{\xi}(\sigma)  \text{ for } \sigma \in H(\bA),
\gamma \in H(\Q), k=(k_v)_v \in H(R_{\bA}),
\end{equation*}
where we denote by $H(R_{\bA})$ the group of adelic isometries of the
lattice $R$ on $D$.
\end{enumerate}
\end{lem}
\begin{proof}
This is easily seen to be a consequence of the fact that the lattice
$I_{ij}=y_i Ry^{-1}_j$ is the quaternionic conjugate of the lattice
$I_{ji}=y_jRy_i^{-1}$ and that quaternionic conjugation is an
element of the (global) orthogonal, but not of the special
orthogonal group of $(D,n)$.
\end{proof}

As in \cite{BS1} we need to show that $\xi$ is proportional to the
function $\xi_{\phi_1,\phi_2}:(y_i,y_j)\mapsto \phi_1(y_i)\otimes
\phi_2(y_j)+\phi_2(y_i)\otimes \phi_1(y_j)$ that appears in our
formula for the Yoshida lifting, and to determine the factor of
proportionality occurring.

\begin{lem}\label{liftback_relation}
With notations as in the previous lemma one has
\begin{equation*}
\langle F,\theta_{ij,\rho}\rangle_{\textup{Pet}}=c_5 {\mathcal P}(
\phi_1(y_i)\otimes \phi_2(y_j)+\phi_2(y_i)\otimes \phi_1(y_j))
\end{equation*}
and
\begin{equation*}
\langle F,F\rangle_{\textup{Pet}}=c_5\langle {\mathcal
P}(\phi_1\otimes \phi_2),{\mathcal P}(\phi_1\otimes \phi_2)\rangle,
\end{equation*}
with some constant $c_5 \ne 0$, where the latter inner product is
the natural inner product on ${\mathcal H}_q(\rho)$-valued functions
on $D_{\bA}^\times \times D_{\bA}^\times$ satisfying the usual
invariance properties under $R_{\bA}^\times$ and $D_{\bQ}^\times$,
which is defined by
\begin{equation*}
\langle {\mathcal P}(\phi_1\otimes \phi_2),{\mathcal
P}(\phi_1\otimes
\phi_2)\rangle=\sum_{i,j=1}^r\frac{1}{e_ie_j}\langle {\mathcal P}
(\phi_1(y_i)\otimes \phi_2(y_j)) , {\mathcal P} (\phi_1(y_i)\otimes
\phi_2(y_j))\rangle_{{\mathcal
    H}_q(\rho)}.
\end{equation*}
\end{lem}
\begin{proof} The proof proceeds in essentially the same way as in
  \cite{BS1}: We notice
first that the space of all $\xi$ with the symmetry property
mentioned (or equivalently the space of functions $\tilde{\xi}$ on
$H(\bA)$ with the invariance property given) has a basis consisting
of the $\xi_{\phi_1,\phi_2}=\xi_{\phi_2,\phi_1}$, where
$(\phi_1,\phi_2)$ runs through the pairs of eigenforms in $(\cA
(D_{\bA}^{\times}, (R_{\bA})^{\times},\tau_1)\times(\cA
(D_{\bA}^{\times}, (R_{\bA})^{\times},\tau_2) $ and where the pairs
are unordered if $\nu_1=\nu_2$.

The Hecke operators $T'_i(p)$ on the spaces $\cA(D_{\bA}^\times,
R_{\bA}^\times,\tau_i)$ (for $i=1,2$) via Brandt matrices described
above induce Hecke operators $\hat{T}(p)$ on the space of $\xi$ as
above that are given by
\begin{equation*}
\xi\vert\hat{T}(p)(y_i,y_j)=\sum_{k=1}^r\tilde{B}^{({\rm
    right})}_{jk}(p)\xi(y_i,y_k)+\sum_{l=1}^r\tilde{B}^{({\rm
    left})}_{il}(p)\xi(y_l,y_j),
\end{equation*}
where for $\nu_1>\nu_2$ we let $\tilde{B}^{({\rm right})}_{jk}(p)$
act on $U=U^{(0)}_{\nu_1}\otimes U^{(0)}_{\nu_2}\oplus
U^{(0)}_{\nu_2}\otimes U^{(0)}_{\nu_1}$ via $\id \otimes
B^{\nu_2}_{jk}(p)\oplus \id \otimes B^{\nu_1}_{jk}(p)$ and
$\tilde{B}^{({\rm left})}_{il}(p)$  as $B^{\nu_1}_{jk}(p)\otimes \id
\oplus B^{\nu_2}_{jk}(p)\otimes \id$, and where for $\nu_1=\nu_2$
the action of $\tilde{B}^{({\rm left})},\tilde{B}^{({\rm right})}$
on $U=U^{(0)}_{\nu_1}\otimes U^{(0)}_{\nu_2}$ is simply the action
of the Brandt matrix on the respective factor of the tensor product.

In the same way as sketched in \cite[10 b)]{BS1} we obtain then
(using the calculations of Hecke operators from \cite{Y1,Y2}) first
\begin{equation*}
\langle F,\theta_{ij,\rho}\vert T(p)\rangle_{\text{Pet}}=\xi\vert
\hat{T}(p)(y_i,y_j).
\end{equation*}

Since, again by Yoshida's computations of Hecke operators (see also
\cite{BS3}), we know that $F$ is an eigenfunction of $T(p)$ with
eigenvalue $\lambda_p(f)+\lambda_p(g)$,  this implies that $\xi$ is
an eigenfunction with the same eigenvalue for $\hat{T}(p)$. A
computation that uses the eigenfunction property of $\phi_1, \phi_2$
for the action of the Hecke operators on the spaces $\cA
(D_{\bA}^{\times}, (R_{\bA})^{\times},\tau_1), \cA
(D_{\bA}^{\times}, (R_{\bA})^{\times},\tau_2)$ shows that the same
is true for the function $\xi_{\phi_1,\phi_2}$.

Since $\phi_1, \phi_2$ are in the essential parts of  $\cA
(D_{\bA}^{\times}, (R_{\bA})^{\times},\tau_1), \cA
(D_{\bA}^{\times}, (R_{\bA})^{\times},\tau_2) $, their eigenvalue
systems occur with strong multiplicity one  in these spaces, and as
in Section 10 of \cite{BS1} we can conclude that that $\xi$ and
$\xi_{\phi_1,\phi_2}$ are indeed proportional, i.e., we have
\begin{equation*}
\langle F,\theta_{ij,\rho}\rangle_{\text{Pet}}=c_5 {\mathcal P}(
\phi_1(y_i)\otimes \phi_2(y_j)+\phi_2(y_i)\otimes \phi_1(y_j))
\end{equation*}
with some constant $c_5 \ne 0$.

From this we see:
\begin{equation*}
\begin{split}
\langle F,F\rangle_{\textup{Pet}}&= \langle F,
\sum_{i,j=1}^r\frac{1}{e_ie_j}\langle {\mathcal P}
(\phi_1(y_i)\otimes \phi_2(y_j)) ,\theta_{ij,\rho}\rangle_{{\mathcal
    H}_q(\rho)}\rangle_{\textup{Pet}}\\
&=\sum_{i,j=1}^r\frac{c_5}{e_ie_j}\langle {\mathcal P}
(\phi_1(y_i)\otimes \phi_2(y_j)) , {\mathcal P} (\phi_1(y_i)\otimes
\phi_2(y_j)+\phi_2(y_i)\otimes \phi_1(y_j))\rangle_{{\mathcal
    H}_q(\rho)}\\
&=c_5\langle {\mathcal P} (\phi_1\otimes \phi_2),{\mathcal P}
(\phi_1\otimes \phi_2)\rangle.
\end{split}
\end{equation*}

\end{proof}
In order to compute the constant $c_5$ we will first need the
generalization of Lemma 9.1 of \cite{BS1} to the present situation:
\begin{lem}\label{orthogonality}
\begin{enumerate}
\item If $\Lambda$ is a lattice on some quaternion algebra $D'$ with
  $n(\Lambda) \subseteq \Z$, of
  level dividing $N$, and with
$\disc(\Lambda) \ne N^2$ the theta series $\theta_{\Lambda,\rho}$ is
orthogonal to all Yoshida lifts $Y^{(2)}(\phi_1,\phi_2)$ of level
$N$.
\item If  $\Lambda$ is a lattice on some quaternion algebra $D'\ne D$ with
  $n(\Lambda) \in \Z$, of
  level $N$ , and with
$\disc(\Lambda) =N^2$ the theta series $\theta_{\Lambda,\rho}$ is
orthogonal to all Yoshida lifts $Y^{(2)}(\phi_1,\phi_2)$ of level
$N$ associated to $D$.
\end{enumerate}
\end{lem}
\begin{proof}
The proof of  Lemma 9.1 of \cite{BS1} unfortunately contains some
misprints: In line 4 on p. 81 the minus sign in front of the whole
factor should not be there and the exponent at $p$ should be
$n(n+1)/2$ (which is equal to $3$ in our present situation), in line
5 the exponent at $p$ should be $n(n-1)/2$ (hence $1$ in our case),
in line 9 the factor $p$ in the right hand side of the equation
should be omitted, and in line 14 the exponent at $p$ should be $1$
instead of $3$.

Apart from these corrections the argument given there carries over
to our situation unchanged. In particular, the results from Section
7 of \cite{BS1} that were used in the proof of that lemma remain
true and their proof carries over if one uses the reformulation of
Evdokimov's result from \cite{evdo} sketched in Section 4 of
\cite{BS3}.
\end{proof}

We recall from \cite{BS1} that we  have
\begin{equation*}
{\mathcal E}_2^{(4)}(Z_1,Z_2)=\sum_{r=1}^t \alpha_r\sum_{\{K_r\}}
\frac{1}{|O(K_r)|}\theta_{K_r}^{(2)}(Z_1)\theta_{K_r}^{(2)}(Z_2),
\end{equation*}
where we denote by $L_1, \ldots, L_t$ representatives of the genera
of lattices of rank 4, square discriminant and level dividing
$N=N_1N_2$, the summation over $\{K_r\}$ runs over a set of
representatives of the isometry classes in the genus of $L_r$ and
$\alpha_r$ are some constants that are explicitly determined in
\cite{BS1}.

Hence by (\ref{theta_gegenbauer_defi}) we obtain
\begin{equation*}
\begin{split}
&\left({\mathbb D}_2\left(\frac{k'-k}{2}-2,k-2\right)({\mathcal E}_2^{(4)})\right)(Z_1,Z_2)\\
&=c_3\sum_{r=1}^t \alpha_r\sum_{\{K_r\}} \frac{1}{|O(K_r)|}
\sum_{({\bf x}_1,{\bf x}_2)\in K_r^2 \times K_r^2} P_{\rm Geg}({\bf
x}_1,{\bf x}_2)\exp(2\pi i \tr(q({\bf x}_1)Z_1+q({\bf x}_2)Z_2))
\end{split}\end{equation*}
with $c_3=(2\pi i)^{k'-k}$, and similarly for the Eisenstein series
${\mathcal F}_2^{(4)}$ attached to the cusp zero, with the
$\alpha_r$ replaced by $\beta_r$ as in \cite{BS1}.

The reproducing property of $P_{\rm Geg}$ implies then
\begin{equation*}\begin{split}
\sum_{({\bf x}_1,{\bf x}_2) \in K_r^2\times K_r^2} P_{\rm Geg}(&{\bf
x}_1,{\bf x}_2)\exp(2\pi i
\tr(q({\bf x}_1)Z_1+q({\bf x}_2)Z_2))\\
&=\langle\langle \theta_{K,\rho}(Z_1)({\bf
u}_1)\otimes\theta_{K,\rho}(Z_2)({\bf u}_2),P_{\rm Geg}({\bf
u}_1,{\bf u}_2)\rangle_{{\mathcal H}_q(\rho)}\rangle_{{\mathcal
H}_q(\rho)}.
\end{split}\end{equation*}
Using the fact that by Lemma \ref{orthogonality} the Yoshida lifting
$F$ is orthogonal to all  $\theta_{K,\rho}$ where $K$ is not in the
genus of the given Eichler order of level $N_1N_2$ we see that the
part of the sum for ${\mathbb D}({\mathcal F}_2^{(4)})(Z_1,Z_2)$
which contributes to the Petersson product with $F$ can be written
as
\begin{equation*}
c_3\beta_1 \sum_{i,j}\frac{1}{e_ie_j}\langle\langle
\theta_{ij,\rho}(Z_1)({\bf u}_1)\otimes\theta_{ij,\rho}(Z_2)({\bf
  u}_2),P_{\rm Geg}({\bf u}_1,{\bf u}_2)\rangle_{{\mathcal H}_q(\rho)}\rangle_{{\mathcal H}_q(\rho)}.
\end{equation*}
We further recall that by
(\ref{ready_to_use_pullback_for_F_weight_2}) we have
\begin{equation*}
\langle F,{\mathbb D}({\mathcal
F}_2^{(4)})(\ast,-\bar{w})\rangle_{\rm
    Pet}=c_4L^{(N)}\left(f\otimes g,\frac{k+k'}{2}\right)F(w)
\end{equation*}
with
\begin{equation}\label{c4_equation}c_4=\lambda \prod_{p\mid N} (1-p^{-1}) \Lambda_N(1)
\frac{1}{\zeta^{(N)}(3)\zeta^{(N)}(4)\zeta^{(N)}(2)}\frac{
C_2\left(\frac{k'-k}{2},k-2\right)}
{C_{2+\frac{1}{2}}\left(\frac{k'-k}{2},k-2\right)} \cdot
\gamma_2\left( \frac{k'-k}{2},k-2,\frac{1}{2}\right).
\end{equation}
\begin{prop}\label{petnorm}
With notations as above we have
\begin{equation*}
\langle F,F\rangle_{\rm Pet}=\frac{c_4}{
2c_3\beta_1}L^{(N)}\left(f\otimes g,\frac{k+k'}{2}\right)\langle
{\mathcal P}(\phi_1\otimes \phi_2),{\mathcal P}(\phi_1\otimes
\phi_2)\rangle.
\end{equation*}
\end{prop}
\begin{proof}
From what we saw above and using Lemma \ref{liftback_relation} we
get
\begin{equation*}\begin{split}
&\langle F,{\mathbb D}({\mathcal
F}_2^{(4)})(\ast,-\bar{Z})\rangle_{\rm
    Pet}\\
&=c_3\beta_1 \sum_{i,j}\frac{1}{e_ie_j}\langle\langle F(\ast),
\theta_{ij,\rho}(-\bar{Z})({\bf u}_1)\otimes\langle
\theta_{ij,\rho}(\ast)({\bf
  u}_2)\rangle_{\textup{Pet}},P_{\rm Geg}({\bf u}_1,{\bf
  u}_2)\rangle_{{\mathcal H}_q(\rho)}\rangle_{{\mathcal H}_q(\rho)}\\
&=c_5 c_3 \beta_1 \sum_{i,j}\frac{1}{e_ie_j}\langle\langle
\overline{\theta_{ij,\rho}(-\bar{Z})}\otimes{\mathcal
  P}(\phi_1(y_i)\otimes\phi_2(y_j)+\phi_2(y_i)\otimes\phi_1(y_j)),P_{\rm Geg}\rangle_{{\mathcal
    H}_q(\rho)}\rangle_{{\mathcal H}_q(\rho)}\\
&=c_5 c_3 \alpha_1
\sum_{i,j}\frac{1}{e_ie_j}\langle\theta_{ij,\rho}(Z),{\mathcal
  P}(\phi_1(y_i)\otimes\phi_2(y_j)+\phi_2(y_i)\otimes\phi_1(y_j))\rangle_{{\mathcal
    H}_q(\rho)}\\
&=2c_3 c_5\beta_1F.
\end{split}\end{equation*}
Comparing with
\begin{equation*}
\langle F,{\mathbb D}({\mathcal
F}_2^{(4)})(\ast,-\bar{Z}))\rangle_{\rm
    Pet}=c_4L^{(N)}\left(f\otimes g,\frac{k+k'}{2}\right)F(Z)
\end{equation*}
we obtain $$c_5=\frac{c_4L^{(N)}(f\otimes
  g,\frac{k+k'}{2})}{2\beta_1c_3},$$
which together with Lemma \ref{liftback_relation} yields the
assertion.
\end{proof}
In order to make use of the above proposition in the next section we
will also need to compare $\langle {\mathcal P}(\phi_1\otimes
\phi_2),{\mathcal P}(\phi_1\otimes \phi_2)\rangle$ with $\langle
\phi_1,\phi_1 \rangle \langle \phi_2,\phi_2\rangle$, where we have
$\langle \phi_\mu,\phi_\mu \rangle=\sum_{i=1}^r\frac{\langle
\phi_\mu(y_i),
  \phi_\mu(y_i)\rangle_\mu}{e_i}$ for $\mu=1,2$, with $\langle
\quad,\quad\rangle_\mu$ denoting the (suitably normalized, see
below) scalar product on $U_{\nu_\mu}^{(0)}$. As always we denote by
$B(x,y)=\tr(x\bar{y})$ the symmetric bilinear form associated to the
quaternionic norm form.
\begin{lem}\label{gegenbauer_compare}
Write $\widetilde{G_a^{(\nu)}}(x)=(B(a,x))^\nu$ for $a,x \in
D_\C^{(0)}:=D_\infty^{(0)}\otimes \C$ with $n(a)=0$ and let
$\nu_1\ge \nu_2$. Then
\begin{enumerate}
\item 
\begin{equation}\begin{split}
{\mathcal P}(\widetilde{G_a^{(\nu_1)}}\otimes&
\widetilde{G_a^{(\nu_2)}})(d_1,d_2)(X_1,X_2)\\&=\frac{\nu_1!}{\nu_2!}
(B(a,d_1)X_1+B(a,d_2)X_2)^{2\nu_2}
\widetilde{G_a^{(\nu_1-\nu_2)}}(\Im(d_1\overline{d_2})).
\end{split}
\end{equation}
\item For $a\in
D_\C^{(0)}$ as above there is $b \in D_\C:=D\otimes \C $ with $ab=0,
a\bar{b}=a, n(b)=0$, and for such a $b$ we have
\begin{equation}\begin{split}
\lim_{\lambda \to 0}\frac{1}{\lambda^{\nu_1-\nu_2}}P_{\rm
  Geg}&((a,a+\lambda b),(d_1,d_2))(Y_1,Y_2,X_1,X_2)\\&
=c_6(B(a,d_1)X_1+B(a,d_2)X_2)^{2\nu_2}\widetilde{G_a^{(\nu_1-\nu_2)}}(\Im(d_1\overline{d_2})(Y_1+Y_2)^{2\nu_2}.
\end{split}
\end{equation}
with $c_6=C_2(\nu_1-\nu_2,2\nu_2)$ as in Lemma
\ref{combinatoricslemma}.
\end{enumerate}
\end{lem}
\begin{proof}
  \begin{enumerate}
  \item From the  formula for the map ${\mathcal P}$ in equation
    (\ref{psi_definition}) we get
\begin{equation*}
\begin{split}
{\mathcal P}(G_a^{(\nu_1)}\otimes&
G_a^{(\nu_2)})(d_1,d_2)(X_1,X_2)\\&=\frac{\nu_1!}{\nu_2!}
(B(a,(d_1X_1+d_2X_2)a(\overline{d_1}X_1+\overline{d_2X_2})))^{\nu_2}
G_a^{(\nu_1-\nu_2)}(\Im(d_1\overline{d_2})).
\end{split}
\end{equation*}
Using $\bar{a}=-a, a^2=0$ and $xa=a\bar{x}-B(a,x)$ for $x \in D_\C$
we get $B(a,ya\bar{x})=B(a,x)B(a,y)$ for $x,y \in D_\C$.
We extend this identity to the polynomial ring, insert for $x,y$ one
of $d_1X_1,d_2X_2$ and  obtain
$B(a,(d_1X_1+d_2X_2)a(\overline{d_1}X_1+\overline{d_2X_2}))=(B(a,d_1)X_1+B(a,d_2)X_2)^2$,
which yields the assertion.
\item For simplicity we identify $D_\C$ with the matrix ring $M_2(\C)$
  and fix $a=\bigl(\begin{smallmatrix}0&1\\0&0\end{smallmatrix}\bigr),
  b=\bigl(\begin{smallmatrix}1&0\\0&0\end{smallmatrix}\bigr)$ (we will
  need this Lemma only for one particular choice of $a,b$). Equation
  (\ref{nose}) in Lemma
  \ref{combinatoricslemma} gives us
  \begin{equation*}\begin{split}
P_{\rm
  Geg}&((a,a+\lambda b),(d_1,d_2))(Y_1,Y_2,X_1,X_2)\\& =
c_6\biggl((Y_1,Y_2)
\begin{pmatrix}
  B(a,d_1)&B(a,d_2)\\ B(a+\lambda b,d_1)&B(a+\lambda b,d_2)
\end{pmatrix}
\begin{pmatrix}
  X_1\\X_2
\end{pmatrix}\biggr)^{2\nu_2}\\&\times \det\bigl(\begin{pmatrix}B(a,d_1)&B(a,d_2)\\
B(a+\lambda b,d_1)&B(a+\lambda b,d_2)
\end{pmatrix} \bigr)^{\nu_1-\nu_2}.  \end{split}
  \end{equation*}
Dividing by $\lambda^{\nu_1-\nu_2}$ and taking the limit for
$\lambda \to 0$ we get  \begin{equation*}
c_6((Y_1+Y_2)(B(a,d_1)X_1+B(a,d_2)X_2))^{2\nu_2}\det\bigl(\begin{pmatrix}B(a,d_1)&B(a,d_2)\\
B(b,d_1)&B(b,d_2)
\end{pmatrix} \bigr)^{\nu_1-\nu_2}.
  \end{equation*}
Computing the determinant for our choice of $a,b$,  writing
$d_1,d_2$ as matrices $\bigl(
\begin{smallmatrix}
  x_1&x_2\\x_3&x_4
\end{smallmatrix}\bigr),\bigl(
\begin{smallmatrix}
  y_1&y_2\\y_3&y_4
\end{smallmatrix}\bigr)$ and using that quaternionic conjugation sends
a matrix  $\bigl(
\begin{smallmatrix}
  x_1&x_2\\x_3&x_4
\end{smallmatrix}\bigr)$ to its classical adjoint  $\bigl(
\begin{smallmatrix}
  x_4&-x_2\\-x_3&x_1
\end{smallmatrix}\bigr)$ one checks that both $\det(\dots)^{\nu_1-\nu_2}$
and $\widetilde{G_a^{(\nu_1-\nu_2)}}(\Im(d_1\overline{d_2}))$ evaluate to
$(x_3y_4-x_4y_3)^{\nu_1-\nu_2}$, which proves the assertion.
  \end{enumerate}
\end{proof}
\begin{prop}\label{normcomparison}
Let  $R_1\in U_{\nu_1}^{(0)},R_2\in U_{\nu_2}^{(0)}$ be given and
let the scalar products $\langle \quad,\quad\rangle_\mu$ on
$U_{\nu_\mu}^{(0)}$ for $\mu=1,2$ be normalized such that the
Gegenbauer polynomial
\begin{equation*}
G^{(\nu_\mu)}(x,y)=\frac{2^{\nu_\mu}}{\Gamma(1/2)}\sum_{j=0}^{[\nu_\mu/2]}
(-1)^j\frac{1}j!(\nu_\mu-2j)!
\Gamma(\nu_\mu-j+\frac{1}{2})(\tr(x\bar{y}))^{\nu_\mu-2j}(n(x)n(y))^j
\end{equation*} (see \cite[p. 47]{BS5}) is the reproducing kernel for
$U_{\nu_\mu}^{(0)}$. Then one has
\begin{equation*}
\langle {\mathcal P}(R_1\otimes R_2),{\mathcal P}(R_1\otimes
R_2)\rangle_{{\mathcal H}_q(\rho)}=c_7\langle
R_1,R_1\rangle_1\langle R_2, R_2\rangle_2
\end{equation*}
with
\begin{equation}\label{c7_equation} c_7=c_6
\frac{\nu_1!}{\nu_2!}\binom{2\nu_1}{\nu_1}\binom{2\nu_2}{\nu_2}=C_2(\nu_1-\nu_2,2\nu_2)\frac{\nu_1!}{\nu_2!}\binom{2\nu_1}{\nu_1}\binom{2\nu_2}{\nu_2}
\end{equation}
with $C_2(\nu_1-\nu_2,2\nu_2)$ explicitly given in Lemma
\ref{combinatoricslemma}.
\end{prop}
\begin{proof}
Since ${\mathcal P}$ is an intertwining map between finite
dimensional irreducible unitary representations of the compact
orthogonal group it is clear that the right hand side and the left
hand side of the asserted equality are proportional. It suffices
therefore to evaluate both sides for a particular choice of
$R_1,R_2$. We choose $R_1=G_a^{(\nu_1)},R_2=G_a^{(\nu_2)}$ with
$G_a^{(\nu_\mu)}(y)=G^{(\nu_\mu)}(a,y)$. The
reproducing property of the Gegenbauer polynomial gives
\begin{equation*}
\langle P_{\rm Geg}(a, a+\lambda b),{\mathcal
P}(G_a^{(\nu_1)}\otimes G_a^{(\nu_2)})\rangle_{{\mathcal
H}_q(\rho)}=({\mathcal P}(G_a^{(\nu_1)}\otimes
G_a^{(\nu_2)}))^c(a,a+\lambda b),
\end{equation*}
where we denote by the exponent $c$ at $({\mathcal
P}(G_a^{(\nu_1)}\otimes G_a^{(\nu_2)}))$ complex conjugation of the
coefficients of this polynomial (in order to avoid confusion with
quaternionic conjugation). With the particular choice of $a,b$ form
the previous lemma we obtain, using $a\bar{b}=a$ and $a^2=0$ and
writing $a^c$ for the vector obtained from $a$ by complex
conjugation of the coordinates with respect to an orthonormal basis
of $D_\infty$,
\begin{equation*}
\begin{split}
  \frac{1}{\lambda^{\nu_1-\nu_2}}&({\mathcal P}(G_a^{(\nu_1)}\otimes 
G_a^{(\nu_2)}))^c(a,a+\lambda
b)(Y_1,Y_2)\\
&=\frac{(Y_1+Y_2)^{2\nu_2}}{\lambda^{\nu_1-\nu_2}}\binom{2\nu_1}{\nu_1}\binom{2\nu_2}{\nu_2}
(B(a^c,a))^{2\nu_2}(B(a^c,\Im(a(\overline{a+\lambda b}))))^{\nu_1-\nu_2}\\
&=\binom{2\nu_1}{\nu_1}\binom{2\nu_2}{\nu_2}(B(a^c,a))^{\nu_1+\nu_2}(Y_1+Y_2)^{2\nu_2}\\
&=(G_a^{(\nu_1)})(a^c)(G_a^{(\nu_2)})(a^c)(Y_1+Y_2)^{2\nu_2}\\
&=\langle G_a^{(\nu_1)}, G_a^{(\nu_1)}\rangle_1 \langle
G_a^{(\nu_2)}, G_a^{(\nu_2)}\rangle_2(Y_1+Y_2)^{2\nu_2}.
\end{split}
\end{equation*}

Inserting the formulas from Lemma \ref{gegenbauer_compare} proves
the assertion.
\end{proof}
\begin{cor}\label{petnorm_explicit}
With notations as in Proposition \ref{petnorm} and ${\mathcal P}$
normalized as above one has
\begin{eqnarray*}
  \langle F,F\rangle_{\rm Pet}&=&\frac{c_4c_7}{
2c_3\beta_1}L^{(N)}\left(f\otimes g,\frac{k+k'}{2}\right)\langle
\phi_1,\phi_1\rangle_1\langle \phi_2,\phi_2\rangle_2\\
&=& \frac{c_4c_7}{
2(2\pi i)^{k'-k}\beta_1}L^{(N)}\left(f\otimes g,\frac{k+k'}{2}\right)\langle
\phi_1,\phi_1\rangle_1\langle \phi_2,\phi_2\rangle_2,
\end{eqnarray*}
with $c_4$ as in (\ref{c4_equation}), $c_7$ as in (\ref{c7_equation}),
and $\beta_1=\beta_1^{(4)}$ as in Corollary 3.2 of \cite{BS1} (with $m=4$,
$r_p(1)=1$ for all $p \mid N$, and
$a_4(N)$ as in Proposition 3.2 of \cite{BS1}).
\end{cor}
Let $F$ be a Yoshida lift of $f$ and $g$ as above  and define
$F_{\mathrm{can}}=\frac{F}{\sqrt{\langle {\mathcal P}(\phi_1\otimes
\phi_2),{\mathcal P}(\phi_1\otimes \phi_2)\rangle}}$. Any rescaling
of $\phi_1,\phi_2$ or $\mathcal{P}$ affects the numerator and
denominator in the same way, so this may be viewed as a canonical
choice of scaling of $F$. We can now express this canonical choice
of $F$ explicitly.
\begin{prop}
Let $\phi_1^{(0)}, \phi_2^{(0)}$ be normalized by $\langle
\phi_1^{(0)},\phi_1^{(0)}\rangle_1=\langle
\phi_2^{(0)},\phi_2^{(0)}\rangle_2=1$and let ${\mathcal P}$ be
normalized as above.
 Then one has
 \begin{equation*}
 F_{\mathrm{can}}=\frac{1}{c_7}Y^2(\phi_1^{(0)},\phi_2^{(0)})
 \end{equation*}
with $c_7$ given explicitly in equation (\ref{c7_equation}).
\end{prop}

Note that the Fourier coefficients of $F_{\mathrm{can}}$ are
algebraic. From the results of \cite{BS5,BS4} it is clear that the
square of the (scalar valued) average over matrices $T$ of fixed fundamental
discriminant $-d$ of the Fourier coefficients $A(F,T)\in W_\rho$  of the Yoshida
lifting $F_{\mathrm{can}}$ is proportional to the product of the
central critical values of the twists with the quadratic character
$\chi_{-d}$ of the $L$-functions of the elliptic modular forms $f$
and $g$; notice that the averaging procedure for the $W_\rho$-valued
Fourier coefficients involves a scalar product of $A(F,T)$ with the
vector $\rho(T^{-1/2}){\bf v}_0$, where ${\bf v}_0$ is an $O_n(\R)$-invariant
vector in $W_\rho$. We can now make this proportionality as explicit as the
result of \cite{BS2} for the scalar valued case.
\begin{prop}\label{twisted-L-propo}
Assume that $\nu_1,\nu_2$ are even and that both $f,g$ have a
$+$-sign in the functional equation. Choose $N_1,N_2$ such that the
(common) Atkin-Lehner eigenvalue $\epsilon_p$ of $f,g$ at $p$ is
$-1$ if and only if $p \mid N_1$. Let $-d<0$ be a fundamental
discriminant with $(\frac{-d}
  {p} )\epsilon_p=1$
for all primes
$p$ dividing  $N_d = N/\gcd(N,d)$. We let $F=F_{\mathrm{can}}$ be
the canonical Yoshida lifting of $f,g$ with respect to $N_1,N_2$ and
put
\begin{equation*}
 a(F,d) \ = {{\sqrt{d}}\over {2}}
\mathop\sum\limits_{{\{T\}}\atop{{\rm \ disc }T=-d}}
{{1}\over{\epsilon(T)}} \ \mathop\int\limits_{T[{\mathbf x}]\leq
1} \ A(F,T)(x_1,x_2)dx_1 dx_2 \,
\end{equation*}
where $A(F,T)$ is the Fourier coefficient at $T$ of  $F$, the
summation is over integral equivalence classes of $T$, and
$\epsilon(T)$ is the number of automorphy (units) of $T$, i.e., the
number of $g \in GL_2(\Z)$ with ${}^tgTg=T$.

Then one has
\begin{equation}\label{twisted-L-formula}
 (a(F,d))^2=c_8\frac{L(1+\nu_1,f)L(,1+\nu_2,g)L(1+\nu_1,f\otimes
   \chi_{-d})L(1+\nu_2,g\otimes \chi_{-d})}{\langle f,f \rangle \langle g,g \rangle}
\end{equation}
with $c_8^{-1}=2^6(\nu_2+1)^2\pi^{2+2\nu_1+2\nu_2}$.
\end{prop}
\begin{proof}
Corollary 4.3 of \cite{BS5} gives
\begin{equation*}
  \left({d \over 4}\right)^{{\nu_1+\nu_2}\over {2}}
\sigma_0(N_d)a(F,d)={{c}\over{2}}
a(\Wscr(\phi_1),d)a(\Wscr(\phi_2),d),
\end{equation*}
where the $a(\Wscr(\phi_\mu),d)$ are the Fourier coefficients of the
Waldspurger liftings
$\Wscr(\phi_\mu)= \mathop{\sum}\limits^r_{j=1} {1\over{e_j}}
\mathop{\sum}\limits_{x\in L_j} \ \phi(y_j)(x)\exp(2\pi i n(x)z)$
associated to the lattices $L_j=D^{(0)}\cap (\Z1+2R_j)$
and where  $c={{(-1)^{\nu_2} 2\pi}\over{2\nu_2+2}}.$ Inserting the
explicit version of Waldspurger's theorem from \cite{kohnen,BS4}
gives the assertion.
\end{proof}
\begin{remar}\label{8.11}
\begin{enumerate}
\item The restrictive conditions on $f,g,N_1,d$ in the proposition are
chosen in order to prevent that $a(F,d)$ becomes zero for trivial
reasons.
\item Since
${{\sqrt{d}}\over {2}}
\mathop\int\limits_{T[{\bf x}]\leq
1} x_1^ix_2^j dx_1 dx_2$
is zero for $i$ or $j$ odd and equal to
\begin{equation*}
\mathop\int\limits_{0}^{\frac{\pi}{2}}
\cos^i(\alpha)\sin^j(\alpha)d\alpha=
\frac{\Gamma(i_1+\frac{1}{2})\Gamma(j_1+\frac{1}{2})}{2\Gamma(i_1+j_1+1)}
\end{equation*}
for even $i=2i_1, j=2j_1$, we have:

\parindent=0pt
If for a prime $\lambda$ not dividing $2\nu_2!$ and some $j \in \N$ one
has $\lambda^j \nmid a(F,d)/\pi$, then there is some $T$ of
discriminant $-d$ such that $\lambda^j$ does not divide all
coefficients of the polynomial $A(F,T)$.
\item With the help of the above proposition for the case
  $\nu_1=\nu_2$ and $f=g$ one could derive an explicit version of
  formula (5.7) of \cite{BS5}. Such an explicit version has been given
  independently by Luo in \cite[(8)]{luo}.
\end{enumerate}
\end{remar}
\section{A congruence of Hecke eigenvalues}
As above, let $f$ and $g$ be cuspidal Hecke eigenforms for
$\Gamma_0(N)$, of weights $k'>k\geq 2$. For critical $k\leq t<k'$,
define $L_{\mathrm{alg}}(f\otimes g,t):=\frac{L(f\otimes
g,t)}{\pi^{2t-(k-1)}\langle f,f\rangle}$. (Alternatively one could
divide by a canonical Deligne period--it makes no difference to the
proposition below.) Let $K$ be a number field containing all the
Hecke eigenvalues of $f$ and $g$. Let $F$ be a Yoshida lift of $f$
and $g$, lying in $S_{\rho}(\Gamma_0^{(2)}(N))$ say, and define as
in the previous section $F_{\mathrm{can}}=\frac{F}{\sqrt{\langle
{\mathcal P}(\phi_1\otimes \phi_2),{\mathcal P}(\phi_1\otimes
\phi_2)\rangle}}$. In fact we have such an $F$ and
$F_{\mathrm{can}}$ for each factorisation $N=N_1N_2$ with an odd
number of prime factors in $N_1$, and we label these $F_i$ and
$F_{i,\mathrm{can}}$ for $1\leq i\leq u$, say. Note that by Lemma
\ref{orthogonality}, these different Yoshida lifts of the same $f$
and $g$ are mutually orthogonal with respect to the Petersson inner
product. Let's say $F=F_1$ arbitrarily.

As in \S 2.1 of \cite{Ar} the operators $T(m)$, for $(m,N)=1$
(generated over $\ZZ$ by the $T(p)$ and $T(p^2)$, see (2.2) of
\cite{Ar}) are self-adjoint for the Petersson inner product, and
commute amongst themselves, so $S_{\rho}(\Gamma_0^{(2)}(N))$ has a
basis of simultaneous eigenvectors for such $T(m)$. Also, these
$T(m)$, acting on elements of $S_{\rho}(\Gamma_0^{(2)}(N))$,
preserve integrality (at any given prime) of Fourier coefficients,
by (2.13) of \cite{Sa}. If $G\in S_{\rho}(\Gamma_0^{(2)}(N))$ is an
eigenform (for the $T(m)$, with $(m,N)=1$), then the Hecke
eigenvalues for $G$ are algebraic integers. This follows from
Theorem I of \cite{We}, which says that the characteristic
polynomial of $\rho_G(\Frob_p^{-1})$ (c.f. \S 4 above) is
$1-\mu_G(p)X+(\mu_G(p)^2-\mu_G(p^2)-p^{k'-2})X^2-p^{k'-1}\mu_G(p)X^3+p^{2(k'-1)}X^4$
(c.f. (2.2) of \cite{Ar}), and that the eigenvalues of
$\rho_G(\Frob_p^{-1})$ are algebraic integers. Moreover, as $p$
varies for fixed $G$, the $\mu_G(p)$ and $\mu_G(p^2)$ generate a
finite extension of $\QQ$.
\begin{prop}\label{maincong} Suppose that $k'-k\geq 6$. Suppose that $\lambda$ is a prime of $K$
such that $ord_{\lambda}\left(L_{\mathrm{alg}}\left(f\otimes
g,\frac{k'+k}{2}\right)\right)>0$ but
$ord_{\lambda}\left(L_{\mathrm{alg}}\left(f\otimes
g,\frac{k'+k}{2}+1\right)\right)=0$, and let $\ell$ be the rational
prime that $\lambda$ divides. Suppose that $\ell\nmid N$ and
$\ell>k'-2$. Assume that there exist a half-integral symmetric
2-by-2 matrix $A$, and an integer $0\leq b\leq k-2$ such that, if
for $1\leq i\leq u$, $a_i$ denotes the coefficient of the monomial
$x^by^{k-2-b}$ in the $A$-Fourier coefficient in
$F_{i,\mathrm{can}}$, then $\ord_{\lambda}(\sum_{i=1}^u a_i^2)\leq
0$. Then there is a cusp form $G\in S_{\rho}(\Gamma_0^{(2)}(N))$, an
eigenvector for all the $T(m)$, with $(m,N)=1$, not itself a Yoshida
lift of the same $f$ and $g$, such that there is a congruence of
Hecke eigenvalues between $G$ and $F$: $$\mu_G(m)\equiv
\mu_F(m)\pmod{\lambda}, \text{ for all $(m,N)=1$.}$$ (We make $K$
sufficiently large to contain the Hecke eigenvalues of $G$.)
\end{prop}
\begin{proof}
Since $k'-k\geq 6$, $\frac{k'-k}{2}-2>0$, so
$\mathbb{D}_4\left(\frac{k'-k}{2}-2,k-2\right)\mathcal{F}_4^{(4)}(Z,W)$
is a cusp form. Let $\{F_1,F_2,\ldots,F_r\}$ be a basis of
$S_{\rho}(\Gamma_0^{(2)}(N))$ consisting of eigenforms for all the
local Hecke algebras at $p\nmid N$, with $F_1,\ldots,F_u$ the
Yoshida lifts of $f$ and $g$, as above.

It is easy to show that
$\mathbb{D}_4\left(\frac{k'-k}{2}-2,k-2\right)
\mathcal{F}_4^{(4)}(Z,W)=\sum_{i,j=1}^r
c_{i,j}F_i(Z)F_j(W)$, for some $c_{i,j}$. By (\ref{pullback4}),
$c_{1,1}$ is equal to the right hand side of (\ref{pullback4}),
divided by $F(w)\langle F,F\rangle$, and $c_{1,j}=0$ for $j\neq 1$.
Similarly for all the $c_{i,i}$ for $1\leq i\leq u$. Using
Proposition \ref{petnorm}, we find
\begin{equation}\label{const}
c_{1,1}=c'\frac{L_{\mathrm{alg}}\left(f\otimes
g,\frac{k'+k}{2}+1\right)}{L_{\mathrm{alg}}\left(f\otimes
g,\frac{k'+k}{2}\right)\langle {\mathcal P}(\phi_1\otimes
\phi_2),{\mathcal P}(\phi_1\otimes \phi_2)\rangle},
\end{equation}
where
\begin{equation}\label{constt}
c'=\gamma_2\left(4,\frac{k'-k}{2}-2,k-2,0\right)(\pm
N)\Lambda_N(2)\frac{\zeta^{(N)}(2)\,\pi^2}{\zeta^{(N)}(4)\zeta^{(N)}(6)\zeta^{(N)}(4)}\prod_{p\mid
N}\frac{(1-p^{-3})}{(1-p^{-1})}.
\end{equation}
(The last term takes into account the fact that we have passed from
incomplete to complete $L$-functions.)

We now choose $A$ and $b$ as in the statement of the proposition.
Imitating \S 4 of \cite{Ka}, let $\FFF_{4,\rho,A}(Z)$ be the
coefficient of $x_w^by_w^{k-2-b}$ in the coefficient of
$\ee(\Tr(AW))$ in
\newline
$\mathbb{D}_4\left(\frac{k'-k}{2}-2,k-2\right)\mathcal{F}_4^{(4)}(Z,W)$.
Then
\begin{equation}\label{lincomb}\FFF_{4,\rho,A}(Z)=\sum_{i=1}^ue_iF_{i,\mathrm{can}}(Z)+\sum_{i\geq u+1}
e_i' F_i(Z),\end{equation} where, for $1\leq i\leq u$,
$e_i=c'\frac{L_{\mathrm{alg}}\left(f\otimes
g,\frac{k'+k}{2}+1\right)}{L_{\mathrm{alg}}\left(f\otimes
g,\frac{k'+k}{2}\right)}a_i$. Careful checking of all the things
that go into $c'$ shows that it is a rational number, and that it
follows from $\ell>k'-2$ that $\ord_{\ell}(c')\leq 0$. The
coefficients of $\FFF_{4,\rho,A}$ are integral at $\lambda$, by
Remarks 7.1 and 7.2. Given all this, we can apply the method of
Lemma 5.1 of \cite{Ka}, to deduce that there is a congruence
$\pmod{\lambda}$ of Hecke eigenvalues (for all $T(m)$, with
$(m,N)=1$) between $F$ and some other $F_i=G$, say, with $i\geq
u+1$.

In a little more detail, we suppose that no such $G$ exists, so that
for each $u+1\leq i\leq r$ there exists an $m_i$, with $(m_i,N)=1$,
such that if $\mu_{F_i}(m_i)$ is the eigenvalue of $T(m_i)$ on $F_i$
then $\mu_{F_i}(m_i)\not\equiv \mu_F(m_i)\pmod{\lambda}$. (We may
enlarge $K$ to contain all the Hecke eigenvalues for all the $F_i$.)
Applying $\prod_{i=u+1}^r(T(m_i)-\mu_{F_i}(m_i))$ to both sides of
(\ref{lincomb}), we get something on the left that is integral at
$\lambda$. On the right all the $F_i$ terms, for $i\geq u+1$,
disappear, while the remaining terms get multiplied by
$\prod_{i=u+1}^r(\mu_F(m_i)-\mu_{F_i}(m_i))$, which is not divisible
by $\lambda$, so on the right-hand-side the coefficient of
$x_z^by_z^{k-2-b}$ in the coefficient of $\ee(\Tr(AZ))$, namely
$$c'\prod_{i=u+1}^r(\mu_F(m_i)-\mu_{F_i}(m_i))\frac{L_{\mathrm{alg}}\left(f\otimes
g,\frac{k'+k}{2}+1\right)}{L_{\mathrm{alg}}\left(f\otimes
g,\frac{k'+k}{2}\right)}\left(\sum_{i=1}^ua_i^2\right),$$ is
non-integral at $\lambda$, which is a contradiction.
\end{proof}
In this proposition, $L\left(f\otimes g,\frac{k'+k}{2}+1\right)$
plays the r\^ole of any critical value further right than the
near-central one. We chose this next-to-near-central value merely
for definiteness. In fact, the further right the evaluation point,
the less laborious is the calculation of the critical value using
Theorem 2 of \cite{Sh}, but we have managed without too much
difficulty in Example 9.1(3) below. Using Proposition
\ref{twisted-L-propo} and Remark \ref{8.11}(2), we obtain the
following.
\begin{cor}\label{Nprime}
Suppose that $k'-k\geq 6$, with $k/2$ and $k'/2$ odd, that $N$ is
prime, and that the common eigenvalue $\epsilon_N$ for $f$ and $g$
is $-1$. Suppose that $\lambda$ is a prime of $K$ such that
$ord_{\lambda}\left(L_{\mathrm{alg}}\left(f\otimes
g,\frac{k'+k}{2}\right)\right)>0$ but
$ord_{\lambda}\left(L_{\mathrm{alg}}\left(f\otimes
g,\frac{k'+k}{2}+1\right)\right)=0$, with $\ell\nmid N$ and
$\ell>k'-2$, where $\ell$ is the rational prime that $\lambda$
divides. Suppose that there is some fundamental discriminant $-d<0$
such that $\left(\frac{-d}{p}\right)=\epsilon_p$ for all primes $p$
dividing $N_d = N/\gcd(N,d)$, such that
$$\ord_{\lambda}\left(\frac{L(k'/2,f)L(k'/2,f\otimes
   \chi_{-d})}{\pi^{k'}\langle f,f \rangle}\frac{L(k/2,g)L(k/2,g\otimes \chi_{-d})}{\pi^k\langle g,g \rangle}\right)\leq 0.$$
Then there is a cusp form $G\in S_{\rho}(\Gamma_0^{(2)}(N))$, an
eigenvector for all the $T(m)$ with $(m,N)=1$, not a multiple of
$F$, such that there is a congruence of Hecke eigenvalues between
$G$ and $F$:
$$\mu_G(m)\equiv \mu_F(m)\pmod{\lambda}, \text{ for all
$(m,N)=1$.}$$ (We make $K$ sufficiently large to contain the Hecke
eigenvalues of $G$.)
\end{cor}

\subsection{Examples}
\begin{enumerate}
\item When $k=2$ and $k'=4$ (so $j=0$ and $\kappa=3$), one may check
that, for $N=23,29,31,37$ or $43$, the dimension of
$S_3(\Gamma_0^{(2)}(N))$ ($2,4,4,9,14$ respectively, using Theorem
2.2 in \cite{I2}) is the same as that of the subspace spanned by
Yoshida lifts of $f\in S_4(\Gamma_0(N))$ and $g\in
S_2(\Gamma_0(N))$. This appears to leave no room for $G$ (recall
Lemma \ref{endocap}). However, we calculated
$L_{\mathrm{alg}}(f\otimes g,3)$ in the case $N=23$, using Theorem 2
of \cite{Sh} and Stein's tables \cite{St}. (The two choices for $g$
are conjugate over $\QQ(\sqrt{5})$.) For the near-central value,
this calculation involves an Eisenstein series of weight $2$, to
which a non-holomorphic adjustment must be made. The result was that
$L_{\mathrm{alg}}(f\otimes g,3)=32/3$, so there is in fact no
divisor $\lambda$, dividing a large prime $\ell$, for which a
congruence with some $G$ is required.

\item The previous paragraph leaves open the possibility that the
condition $k'-k\geq 6$, in Proposition \ref{maincong}, is purely
technical. However, the following example shows that it is
essential. Let $k=2$ and $k'=6$ (so $j=0$ and $\kappa=4$) and
$N=11$. As is well-known, $S_2(\Gamma_0(11))$ is $1$-dimensional,
spanned by $g=q-2q^2-q^3+\ldots$, for which $\epsilon_{11}=-1$.
Using \cite{St}, $\dim S_6(\Gamma_0(11))=4$, with the
$\epsilon_{11}=-1$ eigenspace $3$-dimensional, spanned by the
embeddings of a newform $f=q+\beta q^2+\ldots$, where
$\beta^3-90\beta+188=0$. The discriminant of this polynomial is
$2^43^319\cdot 239$. Using Theorem 2 of \cite{Sh} we find that
$L_{\mathrm{alg}}(f\otimes g,4)=-\frac{4^5\alpha}{3}$, with
$\Norm(\alpha)=-\frac{17\cdot 76157}{2^43^45^211^219\cdot 239}$. In
fact $\alpha$ is divisible by the prime ideals $(17,\beta+1)$ and
$(76157,\beta+74208)$. We check that $L_{\mathrm{alg}}(f\otimes
g,5)=\frac{4^511}{6}\gamma$, with
$\gamma=\frac{1}{1648383}(784522-12341\beta-3842\beta^2)$, of norm
$\frac{2^83\cdot 5^2}{11^19\cdot 239}$, in which $17$ and $76157$ do
not appear.

The dimension of $S_4(\Gamma_0^{(2)}(11))$ is $7$, from the table in
\S 2.4 of \cite{I2}. This fact was also obtained by Poor and Yuen,
who gave an explicit basis for this space using theta series,
\cite{PY}. We are indebted to D. Yuen for calculating for us a Hecke
eigenbasis, which included the three Yoshida lifts, a non-lift with
rational eigenvalues, and three conjugate non-lifts with eigenvalues
and Fourier coefficients in the same cubic field as $f$ and the
Yoshida lifts. He looked for congruences modulo primes dividing $17$
or $76157$ (or any other large primes), but found that there were
none, though it appears that each Yoshida lift has {\em Fourier
coefficients} (not just Hecke eigenvalues) congruent mod $5$ to
those of a corresponding non-lift (suitably normalised).

\item We should expect that any example of $f$ and
$g$ we look at, with prime level $N$, common $\epsilon_N=-1$,
weights $k'>k\geq 2$ with $k'-k\geq 6$ and $k'/2, k/2$ odd, is very
likely to satisfy the remaining conditions of Corollary
\ref{Nprime}, for some $\lambda$. Here is an explicit example. Let
$N=3, k=6,k'=14$. We have $S_6(\Gamma_0(3))$ spanned by
$g=q-6q^2+9q^3+\ldots$, and $S_{14}(\Gamma_0(3))$ spanned by
$f=q+(-27+6\sqrt{1969}) q^2+729 q^3+\ldots$,
$\overline{f}=q+(-27-6\sqrt{1969})q^2+729q^3+\ldots$ and
$h=q-12q^2-729q^3-8048q^4+\ldots$. For both $f$ and $g$,
$\epsilon_3=-1$. Using Theorem 2 of \cite{Sh} we find that
$L_{\mathrm{alg}}(f\otimes g,10)=\frac{-4^{14}}{4!9!3}\alpha$, with
$\alpha=\frac{-467}{35640}-\frac{2119\sqrt{1969}}{140350320}$,
$\Norm(\alpha)=\frac{7\cdot 271\cdot 461\cdot
653}{2^83^75^211^3179}$. (Note that $1969=11\cdot 179$.) So we may
take $\lambda$ to be an appropriate divisor of $\ell=271, 461$ or
$653$. (All three of these primes split in $\QQ(\sqrt{1969})$.) Also
using Theorem 2 of \cite{Sh}, we find that
$L_{\mathrm{alg}}(f\otimes g,11)=\frac{4^{14}3!}{3\cdot
10!5!}\beta$, with $\beta=-25/(3\sqrt{1969})$, so $\lambda\nmid
L_{\mathrm{alg}}(f\otimes g,11)$. Finally, by direct application of
Theorem 5.6 of \cite{GZ}, we calculate $\frac{L(k/2,g)L(k/2,g\otimes
\chi_{-4})}{\pi^k\langle g,g \rangle}=\frac{2^{12}6}{4!4^{5/2}}$ and
$\frac{L(k'/2,f)L(k'/2,f\otimes
   \chi_{-4})}{\pi^{k'}\langle f,f \rangle}=\frac{2^{27}6!}{12!4^{13/2}}\gamma$,
where $\gamma=13488+\frac{256056}{\sqrt{1969}}$, with
$\Norm(\gamma)=\frac{2^63^35^27\cdot 967751}{11\cdot 179}$. The
product of these is not divisible by $\lambda$ (for any of the three
choices).

It seems though that finding an example where one can directly
observe the congruence guaranteed by Corollary \ref{Nprime} would be
difficult. Already for $k=2, k'=10$ and $N=11$ we have $\dim
S_6(\Gamma_0^{(2)}(11))=31$ (from the table in 7-11 of \cite{Has}).
\item For us, $f$ and $g$ are of level $N>1$, and Yoshida lifts do
not exist at level $1$. However, Bergstr\"om, Faber and van der Geer
have found experimentally what appear to be eleven examples of
congruences of exactly the same shape, but for $f$ and $g$ of level
$1$ \cite{BFvdG}. For example, it appears that there is a genus-$2$
cusp form of level $1$ and weight $\Sym^{20}\otimes\det^5$ such that
$$\mu_G(p)\equiv a_p(f)+p^3a_p(g)\pmod{\lambda},$$
with $\lambda\mid 227$, where $f$ and $g$ are cuspidal Hecke
eigenforms of genus $1$, level $1$ and weights $k'=28, k=22$
respectively. Bergstr\"om et. al. have checked this for $p\leq 17$.
Using Theorem 2 of \cite{Sh}, we have checked that $L({f\otimes g},
25)=\frac{4^{27}\pi^{29}}{108(24!)}\cdot\alpha(f,f),$ with
$\Norm(\alpha)=\frac{7.17.227}{2.3^6.5^4.131.139}.$ In two more
examples, with $(k',k,\ell)=(28,18,223)$ and $(28,20,2647)$, we have
likewise checked that the prime occurring in the modulus of an
apparent congruence also appears in the near-central tensor-product
$L$-value, in accord with the Bloch-Kato conjecture.
\end{enumerate}
\subsection{Higher powers of $\lambda$}
A minor modification of the proof of Proposition \ref{maincong}
gives the following.
\begin{prop}\label{maincong1} Suppose that $k'-k\geq 6$. Suppose that $\lambda$ is a prime of $K$
such that $ord_{\lambda}\left(\frac{L_{\mathrm{alg}}\left(f\otimes
g,\frac{k'+k}{2}\right)}{L_{\mathrm{alg}}\left(f\otimes
g,\frac{k'+k}{2}+1\right)}\right)=n>0$, and let $\ell$ be the
rational prime that $\lambda$ divides. Suppose that $\ell\nmid N$
and $\ell>k'-2$. Assume that there exist a half-integral symmetric
2-by-2 matrix $A$, and an integer $0\leq b\leq k-2$ such that, if
for $1\leq i\leq u$, $a_i$ denotes the coefficient of the monomial
$x^by^{k-2-b}$ in the $A$-Fourier coefficient in
$F_{i,\mathrm{can}}$, then $\ord_{\lambda}(\sum_{i=1}^u a_i^2)\leq
0$. Then there are independent cusp forms $G_1,\ldots,G_r\in
S_{\rho}(\Gamma_0^{(2)}(N))$, eigenvectors for all the $T(m)$ with
$(m,N)=1$, not themselves Yoshida lifts of the same $f$ and $g$,
such that there are congruences of Hecke eigenvalues between the
$G_i$ and $F$:
$$\mu_{G_i}(m)\equiv \mu_F(m)\pmod{\lambda^{s(i)}}, \text{ for all
$(m,N)=1$,}$$ with $\sum_{i=1}^r s(i)\geq n$. (We make $K$
sufficiently large to contain the Hecke eigenvalues of $G$.)
\end{prop}
Modifying the proof of Proposition \ref{selmer}, applying the main
theorem of \cite{U2}, one may show (under similar conditions) that
each $G_i$ contributes an element of order $\lambda^{s(i)}$ to
$H^1_f(\QQ,A_{\lambda}((k'+k-2)/2))$, but it does not show that
these elements are independent. However, using Hecke algebras as in
\cite{U}, it should be possible to show that $\lambda^n$ divides $\#
H^1_f(\QQ,A_{\lambda}((k'+k-2)/2))$, and this is covered by the
approach in \cite{AK}, so we leave it to them.

\end{document}